\pdfoutput=1
\RequirePackage{ifpdf}
\ifpdf 
\documentclass[pdftex]{sigma}
\else
\documentclass{sigma}
\fi

\usepackage{tikz-cd}

\usepackage{adjustbox}
\usepackage{microtype}

\usepackage[mathscr]{eucal}
\usepackage[all]{xy}
\usepackage{mathrsfs}

\DeclareMathAlphabet{\mathpzc}{OT1}{pzc}{m}{it}

\newtheorem{thm}{Theorem}[section]
\newtheorem{cor}[thm]{Corollary}
\newtheorem{lem}[thm]{Lemma}
\newtheorem{prop}[thm]{Proposition}

\theoremstyle{definition}
\newtheorem{defn}[thm]{Definition}
\newtheorem{rem}[thm]{Remark}
\newtheorem{Example}[thm]{Example}
\newtheorem{problem}[thm]{Problem}

\numberwithin{equation}{section}

\begin{document}

\allowdisplaybreaks

\newcommand{\arXivNumber}{2101.01913}

\renewcommand{\PaperNumber}{101}

\FirstPageHeading

\ShortArticleName{Moduli Spaces of Homologically Trivial Parabolic (Higgs) Bundles}

\AuthorNameForHeading{X.~Wen}

\ArticleName{Moduli Spaces of Homologically Trivial Parabolic\\ (Higgs) Bundles on the Projective Line\\ and an Application to the Deligne--Simpson Problem}

\Author{Xueqing WEN}

\Address{Chongqing University of Technology, No.~69, Hongguang Avenue, Chongqing, 400054, P.R.~China}
\Email{\mail{wenxq@cqut.edu.cn}}

\ArticleDates{Received January 21, 2026, in final form October 14, 2026; Published online October 23, 2026}

\Abstract{We establish an isomorphism between the moduli space of homologically trivial parabolic (Higgs) bundles on $\mathbb{P}^1$ and the quiver variety associated to a star-shaped quiver. As an application, we give a Hitchin-theoretic construction of irreducible solutions to the nilpotent additive Deligne--Simpson problem under explicit numerical assumptions.}


\Classification{14H60; 16G20; 34M50}

\section{Introduction}

Parabolic bundles on algebraic curves were introduced by Mehta and Seshadri in \cite{MS80}, providing objects that correspond to unitary representations of the fundamental group of a punctured Riemann surface. The moduli theory of parabolic bundles is particularly useful for giving a~finite-dimensional proof—in the sense of Beauville \cite{B95}—of the Verlinde formula, which computes the dimension of spaces of nonabelian theta functions on a curve. In the works \cite{NR93,Sun00,Sun17,SZh18}, the computation of the Verlinde formula is reduced to computing the dimension of nonabelian parabolic theta functions on the projective line $\mathbb{P}^1$, using the method of degeneration of moduli spaces.

A theorem of Grothendieck states that every vector bundle on $\mathbb{P}^1$ decomposes as a direct sum of line bundles, so the classification of ordinary vector bundles on $\mathbb{P}^1$ is completely understood. The situation becomes more subtle and interesting when one considers parabolic bundles and parabolic Higgs bundles on $\mathbb{P}^1$. Parabolic Higgs bundles were first introduced and studied in \cite{Sim90}. In this paper, we denote by $\mathbf{M}_P$ the moduli space of semistable parabolic bundles of rank $r$ and degree $0$ on $\mathbb{P}^1$, and by $\mathbf{Higgs}^{\circ}_P$ the moduli space of homologically trivial semistable parabolic Higgs bundles. Beyond their role in the study of the Verlinde formula, these moduli spaces can also be related to certain quiver varieties.

A quiver is a finite oriented graph. One can define representations of quivers, and their moduli spaces are called quiver varieties. Quiver varieties are well studied and play an important role in representation theory and many other areas. A connection between quiver varieties and moduli spaces of parabolic Higgs bundles was described by Godinho and Mandini in \cite{GM13}. They established an isomorphism between the moduli space of rank two, homologically trivial parabolic Higgs bundles on $\mathbb{P}^1$ and the quiver variety of a certain star-shaped quiver. This result was later extended by Fisher and Rayan in \cite{FR16}, who constructed a similar isomorphism in arbitrary rank $r$; however, in their setting all parabolic structures consist of a one-dimensional flag, and the weights of the parabolic bundles were not fixed. In \cite{RSch20}, Rayan and Schaposnik further generalized the isomorphism to arbitrary parabolic types, but again without specifying the weights.

We establish, in arbitrary rank and for arbitrary parabolic structures, an isomorphism between the moduli space of parabolic bundles on $\mathbb{P}^1$ with a choice of specified weights and the quiver variety associated to a specific star-shaped quiver. For a star-shaped quiver $Q$, we denote the related quiver varieties by $\mathpzc{R}_{\chi}(\mathbf{v})$ and $\mathfrak{M}_{\chi}(\mathbf{v})$ (see Section~\ref{section 3} for details). Our main result is the following.

\begin{thm}[Theorems \ref{main1} and~\ref{main2}]\label{thm1.1}
Under condition~\eqref{Condition}, there exist isomorphisms ${\mathbf{M}_P \cong \mathpzc{R}_{\chi}(\mathbf{v})}$ and $\mathbf{Higgs}^{\circ}_P \cong \mathfrak{M}_{\chi}(\mathbf{v})$, where the quiver~$Q$, the character~$\chi$, and the dimension vector $\mathbf{v}$ are determined by the parabolic data of the parabolic $($Higgs$)$ bundles.
\end{thm}

This construction is also closely related to Soibelman's description.
In \cite{Soi14,Soi16}, Soibelman studied moduli stacks of parabolic bundles
on $\mathbb P^1$ and their relationship with representations of squid
algebras and quivers. In the degree-zero case considered in this paper, the
condition that the underlying vector bundle is homologically trivial is equivalent to
global generation. After choosing a basis of global sections and
trivializations of the flag subspaces, the corresponding objects can be
described by representations of a star-shaped quiver with injective arrows.
Forgetting these trivializations amounts to taking the quotient of a product
of flag varieties by the diagonal action of $\operatorname{GL}(r)$.

Thus Theorem~\ref{thm1.1} should be viewed as an explicit coarse-moduli/GIT
realization of this relationship in the homologically trivial setting, with
fixed flag types and fixed parabolic weights. The choice of the GIT character
is made explicit, and the Higgs version is obtained by passing to the
cotangent/moment-map description, leading to the doubled star-shaped quiver
variety.

We next explain how the Higgs bundle interpretation above leads to a
Hitchin-theoretic construction of solutions to the nilpotent additive
Deligne--Simpson problem. The construction uses the residue description
of homologically trivial parabolic Higgs bundles together with the
parabolic spectral correspondence; the quiver-variety isomorphism provides
the moduli-theoretic bridge with the usual quiver formulation of the
Deligne--Simpson problem.

\subsection{The nilpotent case of the additive Deligne--Simpson problem}

Our application concerns the (additive) Deligne--Simpson problem, which may be formulated as follows: given $n$ conjugacy classes $\mathfrak{c}_i \subset \mathfrak{gl}_r$, does there exist a tuple of matrices $A_i \in \mathfrak{c}_i$ such that $\sum_{i=1}^n A_i = 0$? This problem is naturally linked to the existence of certain Fuchsian systems on $\mathbb{P}^1$. The Deligne--Simpson problem has been studied extensively by Simpson, Kostov, Crawley-Boevey, and others; see, for instance, \cite{CB03,Kos96,Kos03,Kos04,Sim92a, Soi16}. In particular, Crawley-Boevey \cite{CB03} established a correspondence between solutions of the Deligne--Simpson problem and representations of a~certain star-shaped quiver, thereby providing a complete criterion for the existence of solutions.\looseness=1

Building on Crawley-Boevey's work and the isomorphism in Theorem~\ref{thm1.1}, we connect the solutions of the Deligne--Simpson problem to parabolic Higgs bundles on~$\mathbb{P}^1$. From the viewpoint of parabolic Higgs bundles, one studies their characteristic polynomials and the associated spectral curves (see Section~\ref{subsection 2.3}). This geometric perspective yields a Hitchin-theoretic construction under the numerical assumptions stated below, to the nilpotent case of the Deligne--Simpson problem, independent of Crawley-Boevey's original results.

More precisely, if each conjugacy class $\mathfrak{c}_i$ consists of matrices conjugate to a fixed nilpotent matrix $A_i$ of rank $\gamma_i$ (the so-called nilpotent case of the Deligne--Simpson problem), then we have the following statement.

\begin{thm}[Theorem \ref{DSP}]\label{thm1.2}
	Suppose that either $2r<\sum_{i=1}^n\gamma_i$, or $2r=\sum_{i=1}^n\gamma_i$ and $\mathrm{gcd}(r, \gamma_1, \dots, \gamma_n)=1$. Then the nilpotent case of the Deligne--Simpson problem admits irreducible solutions.
\end{thm}

\begin{rem}	\quad
\begin{itemize}\itemsep=0pt
\item[(1)]
 The result largely overlaps with the existence results of Kostov~\cite{Kos03}, Crawley-Boevey~\cite{CB03}, and Soibelman~\cite{Soi16}. In particular, Crawley-Boevey's theorem gives a complete quiver-theoretic criterion for the additive Deligne--Simpson problem. Our contribution is not a~new complete criterion, but a Hitchin-theoretic construction of irreducible solutions under the stated assumptions. More precisely, we can construct such solutions via line bundles on the normalization of spectral curves; see Example~\ref{example of construction}.
		\item[(2)]
The multiplicative Deligne--Simpson problem has recently been solved in
works by Crawley-Boevey--Hubery~\cite{CBH25} and Shu~\cite{Shu25}.
Another perspective, closely related to the parabolic Hitchin system on $\mathbb{P}^1$, is presented in \cite{LL25}. The Deligne--Simpson problem for irregular $G$-connections on $\mathbb{P}^1$ has also been studied by Jakob and Yun in \cite{JY23}; their work is likewise connected to the moduli space of Higgs bundles on $\mathbb{P}^1$.
	\end{itemize}
\end{rem}

This paper is organized as follows.
Section~\ref{section 2} recalls the definition and basic properties of parabolic (Higgs) bundles. Under a suitable choice of weights (condition~\eqref{Condition}), we give an explicit construction of the corresponding moduli spaces. We also examine the parabolic Hitchin map in this setting and review relevant results from~\cite{SWW22}.
In Section~\ref{section 3}, we review the notion of quivers and the construction of quiver varieties. Focusing on the ``star-shaped'' quiver, we then present a proof of Theorem~\ref{thm1.1}.
In Section~\ref{section 5}, we relate homologically trivial parabolic Higgs bundles to solutions of the nilpotent additive Deligne--Simpson problem. By employing the parabolic Hitchin map and the parabolic BNR correspondence (Theorem~\ref{parabolic BNR}), we obtain a~geometric construction of solutions to the nilpotent case of the Deligne--Simpson problem.

\section[Moduli space of parabolic Higgs bundles over the projective line]{Moduli space of parabolic Higgs bundles\\ over the projective line}\label{section 2}

\subsection[Parabolic bundles on P\^{}1]{Parabolic bundles on $\boldsymbol{\mathbb{P}^1}$}

We work over the field $\mathbb{C}$. Let $\mathbb{P}^1$ be the projective line, and let $I = \{x_1,\dots, x_n\} \subset \mathbb{P}^1$ be a finite subset with $n \geq 3$. Fix a positive integer $K$. A parabolic structure on a rank $r$ vector bundle $E$ over $\mathbb{P}^1$ consists of the following data:
\begin{enumerate}\itemsep=0pt
	\item[(1)] A flag at each point $x \in I$:
	\[
	E|_x = F^0(E_x) \supset F^1(E_x) \supset \cdots \supset F^{\sigma_x}(E_x) = 0 .
	\]
	Set $n_i(x) = \dim F^{\,i-1}(E_x) - \dim F^{\,i}(E_x)$ for $1 \le i \le \sigma_x$.
	
	\item[(2)] A strictly increasing sequence of integers attached to each $x \in I$:
	\[
	0 \le a_1(x) < a_2(x) < \cdots < a_{\sigma_x}(x) < K .
	\]
	These numbers are called the weights.
\end{enumerate}

Equipped with the data above, we call $E$ a parabolic vector bundle of type
\begin{equation}\label{parabolic type}
 \Sigma:=\bigl\{I,K,\{n_i(x)\},\{a_i(x)\}\bigr\}.
\end{equation}
When the parabolic type $\Sigma$ is understood, we simply refer to $E$ as a parabolic vector bundle.

The parabolic degree of $E$ is defined as
\[
\operatorname{pardeg} E = \deg E + \frac{1}{K}\sum_{x \in I}\sum_{i=1}^{\sigma_x} a_i(x)  n_i(x).
\]
A parabolic bundle $E$ is called semistable (resp.~stable) if for every proper nonzero subbundle $F \subset E$ endowed with the induced parabolic structure,
\[
\mu_{\mathrm{par}}(F) := \frac{\operatorname{pardeg} F}{\operatorname{rk} F} \leq \mu_{\mathrm{par}}(E)
\qquad\bigl(\text{resp. } \mu_{\mathrm{par}}(F) < \mu_{\mathrm{par}}(E)\bigr).
\]

The construction of moduli spaces of semistable parabolic bundles is treated in \cite{MS80,Sun17}. In the present work, we focus on the moduli space of rank $r$, degree $0$ semistable parabolic bundles on $\mathbb{P}^1$ with sufficiently small parabolic weights. Under this condition we will give an explicit description of the moduli space.
Following~\cite{GM13}, we make the following definition.
\begin{defn}
 A vector bundle $E$ on $\mathbb P^1$ is called \emph{homologically trivial} if $E\simeq \mathcal O_{\mathbb P^1}^{\oplus r}$.
\end{defn}

By the classification of vector bundles on $\mathbb{P}^1$ and the computation of cohomology of line bundles on $\mathbb{P}^1$, we have the following lemma.

\begin{lem}\label{characterization}
 A vector bundle $E$ on $\mathbb{P}^1$ is homologically trivial if and only if
 \[
 \operatorname{H}^0\bigl(\mathbb{P}^1, E\otimes \mathcal{O}_{\mathbb{P}^1}(-1)\bigr) =\operatorname{H}^1\bigl(\mathbb{P}^1, E\otimes \mathcal{O}_{\mathbb{P}^1}(-1)\bigr)=0.
 \]
\end{lem}

Before proceeding, we impose the following condition on the weights, which will be used throughout the sequel:
\begin{align}\label{Condition}
	\frac{1}{K}\sum_{x\in I}a_{\sigma_x}(x)< \frac{1}{r}.
\end{align}

\begin{lem}\label{trivial bundle}
Assume condition~\eqref{Condition} holds. If $E$ is a semistable parabolic vector bundle of rank~$r$ and degree~$0$, then~$E$ is homologically trivial.
\end{lem}

\begin{proof}
	Let $F \subset E$ be a proper nonzero subbundle. Its parabolic degree is
	\[
	\operatorname{pardeg} F = \deg F + \frac{1}{K}\sum_{x \in I}\sum_{i=1}^{\sigma_x} a_i(x)\, n_i^F(x),
	\]
	where $n_i^F(x)=\dim\!\big(F^{\,i-1}(E_x) \cap F|_x\big) - \dim\!\big(F^{\,i}(E_x) \cap F|_x\big)$.
	The semistability of $E$ implies
	\[
	\frac{\deg F}{\operatorname{rk} F} - \frac{\deg E}{\operatorname{rk} E}
	\le \frac{1}{K}\sum_{x\in I}\sum_{i=1}^{\sigma_x} a_i(x)
	\left(\frac{n_i(x)}{\operatorname{rk} E} - \frac{n_i^F(x)}{\operatorname{rk} F}\right),
	\]
	which, since $\deg E = 0$, simplifies to
	\[
	\frac{\deg F}{\operatorname{rk} F}
	\le \frac{1}{K}\sum_{x\in I}\sum_{i=1}^{\sigma_x} a_i(x)
	\left(\frac{n_i(x)}{\operatorname{rk} E} - \frac{n_i^F(x)}{\operatorname{rk} F}\right).
	\]
	
Condition~\eqref{Condition} guarantees that the right-hand side of the above inequality is strictly smaller than~$1/r$. Consequently, $\deg F \le 0$ for every proper nonzero subbundle~$F$. By Grothendieck's classification of vector bundles on~$\mathbb{P}^1$, this forces $E$ to be homologically trivial, i.e., $E \cong \mathcal{O}_{\mathbb{P}^1}^{\oplus r}$.
\end{proof}

\begin{Example}\samepage We provide a counterexample to Lemma~\ref{trivial bundle} when condition~\eqref{Condition} is not satisfied.
	Let $I = \{x_1,\dots, x_4\}$, $K=2$, and consider the vector bundle
	$E = \mathcal{O}_{\mathbb{P}^1}(-1) \oplus \mathcal{O}_{\mathbb{P}^1}(1)$.
	We endow~$E$ with the following parabolic structure:
\begin{itemize}\itemsep=0pt
		\item[(1)] At each $x_i \in I$ the flag is
		\[
		\mathcal{O}_{\mathbb{P}^1}(-1)|_{x_i} \oplus \mathcal{O}_{\mathbb{P}^1}(1)|_{x_i}
		 \supset \mathcal{O}_{\mathbb{P}^1}(-1)|_{x_i} \supset 0;
		\]
		\item[(2)] The weights are $a_1(x_i)=0$, $a_2(x_i)=1$ for every $x_i$.
	\end{itemize}
	
Condition~\eqref{Condition} clearly fails in this case. We now briefly verify that~$E$ is nevertheless semistable.
	
First, $\mu_{\mathrm{par}}(E)=1$. Consider proper sub-line-bundles of $E$.
When $\mathcal{O}_{\mathbb{P}^1}(1)$ is considered as a~subbundle of~$E$, then it must be a direct summand; with the induced parabolic structure one checks $\mu_{\mathrm{par}}\bigl(\mathcal{O}_{\mathbb{P}^1}(1)\bigr)=1$.
	The bundle $\mathcal{O}_{\mathbb{P}^1}$ cannot be a subbundle of $E$.
	For any $\mathcal{O}_{\mathbb{P}^1}(-n)$ with $n\ge 1$ that injects into $E$, one finds $\mu_{\mathrm{par}}\bigl(\mathcal{O}_{\mathbb{P}^1}(-n)\bigr) \le 1$.
	Hence all proper nonzero subbundles satisfy the semistability inequality, so $E$ is a semistable parabolic bundle of degree $0$ on $\mathbb{P}^1$ that is not homologically trivial.
\end{Example}

\begin{lem}\label{sub trivial bundle}
Assume condition~\eqref{Condition} holds. A homologically trivial parabolic vector bundle $E$ is semistable if and only if for every homologically trivial subbundle $F\subset E$ one has
\[
\mu_{\mathrm{par}}(F) \le \mu_{\mathrm{par}}(E).
\]
\end{lem}

\begin{proof}
One implication is immediate. For the converse, suppose every homologically trivial subbundle satisfies the inequality, and let $F\subset E$ be an arbitrary subbundle that is \emph{not} homologically trivial. Then $\deg F < 0$. Using $\deg E = 0$, and by condition~\eqref{Condition}, we obtain
\begin{align*}
		\mu_{\mathrm{par}}(F)-\mu_{\mathrm{par}}(E)&=\frac{\mathrm{deg}F}{\mathrm{rk}F}-\frac{1}{K}\sum_{x\in I}\sum_{i=1}^{\sigma_x}a_i(x)\left(\frac{n_i(x)}{\mathrm{rk}E}-\frac{n_i^F(x)}{\mathrm{rk}F}\right)\\
 & =\frac{\mathrm{deg}F}{\mathrm{rk}F}+\frac{1}{K}\sum_{x\in I}\sum_{i=1}^{\sigma_x}a_i(x)\left(\frac{n_i^F(x)}{\mathrm{rk}F}-\frac{n_i(x)}{\mathrm{rk}E}\right)
		 \leq \frac{\mathrm{deg}F}{\mathrm{rk}F}+\frac{1}{r}<0.\tag*{\qed}
\end{align*}\renewcommand{\qed}{}
\end{proof}

\begin{Example}\label{key example}
We now give a counterexample to Lemma~\ref{sub trivial bundle} when condition~\eqref{Condition} is violated.
	
Let $V$ be a two-dimensional vector space and identify $\mathbb{P}(V)\cong\mathbb{P}^1$.
Consider the trivial rank~$2$ bundle $E = \mathbb{P}(V)\times V$.
It contains a canonical line subbundle $\mathcal{L}$ whose fiber over a point $[l]\in\mathbb{P}(V)$ (corresponding to a line $l\subset V$) is precisely $l\subset V = E|_{[l]}$.
	
Choose four distinct points $[l_1],\dots,[l_4]\in\mathbb{P}(V)$ and set $K=4$.
Equip $E$ with the parabolic structure defined by
	\begin{itemize}\itemsep=0pt
		\item[(1)] at each $[l_i]$ the flag $E|_{[l_i]}=V \supset l_i \supset 0$;
		\item[(2)] weights $a_1([l_i])=0$, $a_2([l_i])=3$ for every $i$.
	\end{itemize}
	This choice does not satisfy condition~\eqref{Condition}.
	
	We claim that every homologically trivial subbundle $F\subset E$ satisfies $\mu_{\mathrm{par}}(F) < \mu_{\mathrm{par}}(E)$, whereas $\mu_{\mathrm{par}}(\mathcal{L}) > \mu_{\mathrm{par}}(E)$, showing that condition~\eqref{Condition} cannot be omitted.
	
	First, $\mu_{\mathrm{par}}(E) = \frac{3}{2}$.
	A homologically trivial subbundle $F$ corresponds to a line $W\subset V$.
	A case-by-case check (according to whether $W$ coincides with some $l_i$ or not) shows that $\mu_{\mathrm{par}}(F) < \frac{3}{2}$.
	On the other hand,
	\[
	\mu_{\mathrm{par}}(\mathcal{L}) = \deg\mathcal{L} + \frac{1}{4}\sum_{i=1}^4 3
	= -1 + \frac{12}{4} = 2 > \frac{3}{2},
	\]
	so $\mathcal{L}$ violates the semistability condition.
	Hence the conclusion of Lemma~\ref{sub trivial bundle} fails without condition~\eqref{Condition}.
\end{Example}

We now construct the moduli space of semistable parabolic vector bundles of rank $r$ and degree $0$ under condition~\eqref{Condition}.

By Lemma~\ref{trivial bundle}, every such semistable parabolic vector bundle is isomorphic, as a vector bundle, to $\mathcal{O}_{\mathbb{P}^1}^{\oplus r}$.
Set $V = \operatorname{H}^0\bigl(\mathbb{P}^1, \mathcal{O}_{\mathbb{P}^1}^{\oplus r}\bigr)$.
All possible choices of flags at the marked points on $\mathcal{O}_{\mathbb{P}^1}^{\oplus r}$ are parametrized by the product of partial flag varieties
\[
\mathbf{F}:=\prod_{x\in I} \operatorname{Flag}\bigl(V, \vec{\gamma}(x)\bigr),
\]
where $\operatorname{Flag}(V,\vec{\gamma}(x))$ is the variety of partial flags in $V$ with dimension vector
\[
\vec{\gamma}(x)=\bigl(\gamma^1(x), \gamma^2(x), \dots , \gamma^{\sigma_x-1}(x)\bigr),
\qquad \gamma^i(x)=\sum_{j=i+1}^{\sigma_x}n_j(x).
\]
The group $\operatorname{SL}(V)$ acts diagonally on $\mathbf{F}$; two points of $\mathbf{F}$ correspond to isomorphic parabolic bundles precisely when they lie in the same $\operatorname{SL}(V)$-orbit.

We polarize the $\operatorname{SL}(V)$-action on $\mathbf{F}$ by the tuple
\[
\prod_{x\in I}\bigl(d_{1}(x),\dots, d_{\sigma_x-1}(x)\bigr),
\qquad d_i(x)=a_{i+1}(x)-a_i(x).
\]
By the Hilbert--Mumford criterion, a point
\[
q=\prod_{x\in I}\bigl(q_1(x)\supset \cdots \supset q_{\sigma_x-1}(x)\bigr)\in \mathbf{F}
\]
is GIT semistable if and only if for every proper nonzero subspace $W\subset V$,
\begin{align*}
	\frac{\sum_{x\in I}\sum_{i=1}^{\sigma_x-1}d_{i}(x)\dim\!\big(W\cap q_i(x)\big)}
	{\sum_{x\in I}\sum_{i=1}^{\sigma_x-1}d_{i}(x)\dim\!\big(q_i(x)\big)}
	\le \frac{\dim W}{r}.
\end{align*}
After rearranging and letting $E$ be the parabolic bundle corresponding to $q$, one checks that this inequality is equivalent to
\[
\frac{\operatorname{pardeg}(W\otimes \mathcal{O}_{\mathbb{P}^1})}{\dim W}
\le \frac{\operatorname{pardeg} E}{r}.
\]
Lemma~\ref{sub trivial bundle} then shows that GIT semistability coincides with parabolic semistability. Hence we obtain the following statement.

\begin{prop}\label{moduli of bundle}
Under condition~\eqref{Condition}, the moduli space of semistable parabolic vector bundles of rank $r$ and degree $0$ on $\mathbb{P}^1$ with type $\Sigma$ as in~\eqref{parabolic type}, denoted by $\mathbf{M}_P(\Sigma)$, is isomorphic to the GIT quotient $\mathbf{F}/\!/\operatorname{SL}(V)$ with the polarization described above.
\end{prop}

We also use $\mathbf{M}_P$ to denote the moduli space if there is no confusion.

\begin{rem}\label{key remark}
When condition~\eqref{Condition} is not satisfied, Example~\ref{key example} shows that the moduli space may not be realizable as a GIT quotient of a product of partial flag varieties. Indeed, the space of weights for parabolic bundles exhibits a ``wall and chamber'' structure \cite{BH95, Tha96}. Our condition~\eqref{Condition} ensures that the chosen weights are sufficiently small, i.e., they lie in a chamber close to the origin. As shown in \cite[Proposition~3.7]{MY16}, if the weights are varied, the moduli space of semistable parabolic vector bundles can become a blow-up of $\mathbf{F}/\!/\operatorname{SL}(V)$ at a point.
\end{rem}

\subsection[Parabolic Higgs bundles on P\^{}1]{Parabolic Higgs bundles on $\boldsymbol{\mathbb{P}^1}$}

We now turn to parabolic Higgs bundles on $\mathbb{P}^1$ and their moduli.
Let $D_I = \sum_{x\in I}x$ be a reduced effective divisor on~$\mathbb{P}^1$. A parabolic Higgs bundle is a pair $(E,\phi)$, where $E$ is a parabolic vector bundle and
$\phi\colon  E\longrightarrow E\otimes\omega_{\mathbb{P}^1}(D_I)$
is a strongly parabolic Higgs field. Throughout this paper, this means
that, for every $x\in I$ and every~$i$,
\[
\operatorname{Res}_x(\phi)\bigl(F^i(E_x)\bigr) \subset F^{i+1}(E_x)\otimes \omega_{\mathbb{P}^1}(D_I)|_x.
\]
In other words, the residue lowers the parabolic filtration by one step. Some authors use the term ``strictly parabolic'' for this condition; we consistently use ``strongly parabolic''.

A parabolic Higgs bundle $(E,\phi)$ is called semistable (resp.~stable) if, for every proper nonzero $\phi$-invariant subbundle $F\subset E$, endowed with the induced parabolic structure, one has
$\mu_{\mathrm{par}}(F)
\leq
\mu_{\mathrm{par}}(E)$
$(\text{resp. }< )$.
Here $\phi$-invariance means that
$\phi(F)\subset F\otimes\omega_{\mathbb{P}^1}(D_I)$,
and $\phi^\prime$ is the restriction
\[
\phi'
:=
\phi|_F
\colon \
F\longrightarrow
F\otimes\omega_{\mathbb{P}^1}(D_I),
\]
which is the induced Higgs field on $F$.

Before constructing the moduli space of parabolic Higgs bundles, we illustrate the definition with an example.

\begin{Example}
	Take $I=\{x_1,\dots,x_4\}\subset\mathbb{P}^1$, $K=16$, and the vector bundle
	$E = \mathcal{O}_{\mathbb{P}^1}(1) \oplus \mathcal{O}_{\mathbb{P}^1}(-1)$.
	Equip $E$ with the parabolic structure defined by
	\begin{itemize}\itemsep=0pt
		\item[(1)] the flag
$\mathcal{O}_{\mathbb{P}^1}(-1)|_{x_i} \oplus \mathcal{O}_{\mathbb{P}^1}(1)|_{x_i}
		 \supset \mathcal{O}_{\mathbb{P}^1}(-1)|_{x_i} \supset 0$
		$(i=1,\dots,4)$;
		\item[(2)] the weights $a_1(x_i)=0$, $a_2(x_i)=1$ for every $x_i$.
	\end{itemize}
	This choice of weights satisfies condition~\eqref{Condition}; consequently, by Lemma~\ref{trivial bundle}, $E$ is not semistable as a parabolic bundle.
	Now fix a non-zero morphism $\phi_1:\mathcal{O}_{\mathbb{P}^1}(1)\to\mathcal{O}_{\mathbb{P}^1}(-1)\otimes\omega_{\mathbb{P}^1}(D_I)$ and define
	\[
	\phi = \begin{bmatrix} 0 & 0 \\ \phi_1 & 0 \end{bmatrix}\colon \ E \longrightarrow E\otimes\omega_{\mathbb{P}^1}(D_I).
	\]
One checks that $\phi$ is a strongly parabolic Higgs field.
Since $\mathcal{O}_{\mathbb{P}^1}(1)$ is not a Higgs subbundle of~$(E,\phi)$, it is straightforward to verify that $(E,\phi)$ is a stable parabolic Higgs bundle.
\end{Example}

Thus, in contrast to the case of parabolic vector bundles, a semistable parabolic Higgs bundle need not have an underlying homologically trivial vector bundle even when condition~\eqref{Condition} is imposed. For this reason we restrict attention to parabolic Higgs bundles whose underlying vector bundle is homologically trivial.

\begin{rem}
	Moduli spaces of semistable parabolic Higgs bundles on a curve $X$ were constructed in \cite{Yo93}; the locus of homologically trivial bundles forms an open subspace of this moduli space by Lemma~\ref{characterization}.
\end{rem}

\begin{lem}\label{lem:par-Higgs-stability}
	Assume condition~\eqref{Condition} holds. A homologically trivial parabolic Higgs bundle $(E,\phi)$ is semistable $($resp.~stable$)$ if and only if for every proper homologically trivial Higgs subbundle $(F,\phi')\subset(E,\phi)$,
$\mu_{\mathrm{par}}(F) \leq \mu_{\mathrm{par}}(E)$ $(\text{resp. }<)$.
\end{lem}

\begin{proof}
	The argument is the same as that of Lemma~\ref{sub trivial bundle}.
\end{proof}

Before constructing the moduli space of semistable homologically trivial parabolic Higgs bundles, we examine parabolic Higgs fields $\phi$ for such bundles more closely. Let $\operatorname{Hom}_{\mathrm{spar}}(E, E\otimes \omega_{\mathbb{P}^1}(D_I))$ denote the space of strongly parabolic Higgs fields on $E$.
Fix an isomorphism $\operatorname{H}^0\bigl(\mathbb{P}^1,E\bigr)\cong V$. Then for each $x\in I$, the filtration on $E|_x$ induces a filtration on~$V$.
Recall that
\[
\phi \in \operatorname{Hom}(E, E\otimes \omega_{\mathbb{P}^1}(D_I))
\cong \operatorname{Hom}_\mathbb{C} \bigl(V, V\otimes \operatorname{H}^0\bigl(\mathbb{P}^1, \omega_{\mathbb{P}^1}(D_I)\bigr)\bigr).
\]

For any $x\in I$, the residue map
\[
\operatorname{Res}_x\colon \ \operatorname{H}^0\bigl(\mathbb{P}^1, \omega_{\mathbb{P}^1}(D_I)\bigr) \longrightarrow \omega_{\mathbb{P}^1}(D_I)|_x
\]
gives a composed map $\operatorname{Res}_x\circ\phi\colon V \to V\otimes \omega_{\mathbb{P}^1}(D_I)|_x$.
A morphism $\phi$ is a parabolic Higgs field if and only if for every $x\in I$ the map $\operatorname{Res}_x\circ\phi$ lowers the filtration on $V$ induced from $E|_x$ by one step.
Write $\operatorname{Hom}^{\mathrm{s.f.}}(V, V\otimes \omega_{\mathbb{P}^1}(D_I)|_x)$ for the space of such maps.

Since $\deg \omega_{\mathbb{P}^1}(D_I)=n-2$, and we have assumed that $n\geq 3$, we have the exact sequence
\[
0 \longrightarrow H^0\bigl(\mathbb{P}^1, \omega_{\mathbb{P}^1}(D_I)\bigr)
\longrightarrow \bigoplus_{x\in I} \omega_{\mathbb{P}^1}(D_I)|_x
\longrightarrow \mathbb{C} \longrightarrow 0 .
\]
The map to $\mathbb{C}$ is the summation over the index set $I$. Choosing the local basis ${\rm d}z/(z-x)$ for each $\omega_{\mathbb{P}^1}(D_I)|_x$, we obtain
\begin{equation}\label{exseq}
	0 \rightarrow \operatorname{Hom}_{\mathrm{spar}}(E, E\otimes \omega_{\mathbb{P}^1}(D_I))
	\hookrightarrow \bigoplus_{x\in I} \operatorname{Hom}^{\mathrm{s.f.}}(V,V)
	\rightarrow \operatorname{Hom}(V,V).
\end{equation}
Here the map to $\operatorname{Hom}(V,V)$ is also the summation over the index set $I$.

So the exact sequence~\eqref{exseq} shows that giving a parabolic Higgs field on a homologically trivial parabolic bundle $E$ is equivalent to giving $n$ linear maps $A_x\colon V\to V$ $(x\in I)$ satisfying certain nilpotency conditions, such that $\sum_{x\in I} A_x = 0$.

We now construct the desired moduli space. Since we consider only homologically trivial bundles, all parabolic structures are again parametrized by $\mathbf{F}$.
Observe that the middle term of the exact sequence~\eqref{exseq} is the cotangent space to $\mathbf{F}$ at a point.
From the discussion above, we obtain a morphism of vector bundles over $\mathbf{F}$,
$\mu_P\colon T^*\mathbf{F} \longrightarrow \mathfrak{sl}(V)^*$,
and $\mathfrak{F}:= \mu_P^{-1}(0)$ parametrizes all parabolic Higgs bundles under consideration.
Thus we obtain the following description.

\begin{prop}\label{HiggsP}
	Under condition~\eqref{Condition}, the moduli space of homologically trivial semistable parabolic Higgs bundles of rank $r$ and degree $0$ on $\mathbb{P}^1$ with type $\Sigma$ as in~\eqref{parabolic type}, denoted as $\mathbf{Higgs}^{\circ}_P(\Sigma)$, is isomorphic to the GIT quotient $\mathfrak{F}/\!/\operatorname{SL}(V)$, with the polarization chosen earlier.
\end{prop}

Similarly, if there is no confusion, we also use $\mathbf{Higgs}^{\circ}_P$ to denote the moduli space.

\begin{rem}\label{key example 1}
	\quad
	\begin{itemize}\itemsep=0pt
		\item[(1)] Consider the parabolic bundle $E$ from Example~\ref{key example} equipped with the zero Higgs field $\phi=0$.
		For every homologically trivial Higgs subbundle $F\subset E$ one has $\mu_{\mathrm{par}}(F) < \mu_{\mathrm{par}}(E)$; nevertheless, $E$ itself is not a semistable parabolic Higgs bundle.
		Hence condition~\eqref{Condition} is also essential in the parabolic Higgs bundle setting.
		
		\item[(2)] For the details of performing the GIT quotient on the variety $\mathfrak{F}$, we refer to \cite{Ni91}.
	\end{itemize}
\end{rem}

\subsection{Parabolic Hitchin map}\label{subsection 2.3}

Let $(E,\phi)$ be a parabolic Higgs bundle. Its characteristic polynomial is defined as
\[
\operatorname{char}(E,\phi)=\lambda^{r}+\alpha_{1}\lambda^{r-1}+\cdots+\alpha_{r-1}\lambda+\alpha_{r},
\]
where $\alpha_i = (-1)^i\operatorname{Tr}(\wedge^{i}\phi) \in \operatorname{H}^{0}\bigl(\mathbb{P}^{1},\omega_{\mathbb{P}^{1}}(D_I)^{\otimes i}\bigr)$.
Writing the polynomial by its coefficients, we may regard
\[
\operatorname{char}(E,\phi)=\alpha=(\alpha_i)_{1\le i\le r}\in
\mathbf{H}:=\prod_{i=1}^{r} \operatorname{H}^{0}\bigl(\mathbb{P}^{1},\omega_{\mathbb{P}^{1}}(D_I)^{\otimes i}\bigr).
\]

Globally, let $\mathbf{Higgs}_P$ denote the moduli space of semistable parabolic Higgs bundles of rank $r$ and degree $0$ on $\mathbb{P}^1$ with type $\Sigma$ as in~\eqref{parabolic type}.
There is an algebraic proper morphism, the parabolic Hitchin map,
$h_P\colon \mathbf{Higgs}_P\longrightarrow \mathbf{H}$,
which sends the S-equivalence class of a semistable parabolic Higgs bundle $(E,\phi)$ to $\operatorname{char}(E,\phi)$.
For a detailed discussion of this construction we refer to \cite{SWW22}.

Because the parabolic Higgs field satisfies $\phi (F^{i}(E_x) )\subseteq F^{i+1} ((E\otimes\omega_{\mathbb{P}^{1}}(D_I))_x )$ at every marked point $x\in I$, the residue of $\phi$ at each $x$ is nilpotent. Consequently, the map $h_P$ is not surjective.

To describe the image of $h_P$, we associate with a fixed parabolic type
\[
\Sigma=\{I,K,\{n_i(x)\},\{a_i(x)\}\},
\]
the following combinatorial numbers for each $x\in I$ and each $1\le j\le r$:
\begin{align}
	&\mu_j(x)
=\#\bigl\{l \mid n_l(x)\ge j,\, 1\le l\le \sigma_x\bigr\},
	\label{muj}\\
&	\varepsilon_j(x)=q
\ \Longleftrightarrow\
	\sum_{t\le q-1}\mu_t(x)<j\le \sum_{t\le q}\mu_t(x).
	\label{eq:epsilonj}
\end{align}
Hence $\{\mu_j(x)\}$ forms the dual partition of $\{n_i(x)\}$. The numbers $\varepsilon_j(x)$ can be read from the Young tableau of the flag type $\{n_i(x)\}$, see \cite[diagram~(3.2)]{SWW22}. Note that $\varepsilon_r(x) = \max\{ n_i(x) \}$.

\begin{prop}[{\cite[Theorems 3.4 and~4.9]{SWW22}}]\label{Hitchin base}
	The image of $h_P$ is contained in the subspace
	\[
	\mathbf{H}_P:=\prod_{j=1}^{r} \operatorname{H}^{0}\biggl(\mathbb{P}^{1},
	\omega_{\mathbb{P}^{1}}^{\otimes j}\otimes\mathcal{O}_{\mathbb{P}^{1}}
	\biggl(\sum_{x\in I}(j-\varepsilon_j(x))x\biggr)\biggr)
	\]
	of $\mathbf{H}$. The induced morphism $h_P\colon \mathbf{Higgs}_P\to\mathbf{H}_P$ is called the parabolic Hitchin map, and $\mathbf{H}_P$ is called the parabolic Hitchin base.
\end{prop}

Denote by $|\omega_{\mathbb{P}^{1}}(D_I)| = \operatorname{Spec}\bigl(\operatorname{Sym}\bigl(\omega_{\mathbb{P}^{1}}(D_I)^{-1}\bigr)\bigr)$ the total space of $\omega_{\mathbb{P}^{1}}(D_I)$.
For any $\alpha = (\alpha_i)_{1\le i\le r} \in \mathbf{H}$, the associated spectral curve $C_{\alpha}$ is constructed as follows.
Each $\alpha_i$ induces a~homomorphism
$
\alpha_i\colon  \omega_{\mathbb{P}^{1}}(D_I)^{-r}\longrightarrow\omega_{\mathbb{P}^{1}}(D_I)^{-(r-i)}$.
Summing these homomorphisms together with the natural inclusion of $\omega_{\mathbb{P}^{1}}(D_I)^{-r}$ into the degree $r$ summand yields a map
$
u\colon   \omega_{\mathbb{P}^{1}}(D_I)^{-r}\longrightarrow\operatorname{Sym} \bigl(\omega_{\mathbb{P}^{1}}(D_I)^{-1}\bigr)$.
Let $\mathscr{J}$ be the ideal generated by the image of $u$. The spectral curve attached to $\alpha$ is defined as
$
C_{\alpha}= \operatorname{Spec} \bigl(\operatorname{Sym} \bigl(\omega_{\mathbb{P}^{1}}(D_I)^{-1}\bigr)\big/ \mathscr{J}\bigr)$.

There is a natural projection $\pi_{\alpha}\colon C_{\alpha}\to\mathbb{P}^{1}$.
If a parabolic Higgs bundle $(E,\phi)$ satisfies $\operatorname{char}(E,\phi)=\alpha$, we also call $C_{\alpha}$ the spectral curve of $(E,\phi)$.
Informally, viewing a parabolic Higgs bundle as a family of linear maps parameterized by $\mathbb{P}^{1}$, the curve $C_{\alpha}$ parameterizes the corresponding eigenvalues.

In \cite{SWW22}, we investigate the generic fiber of the parabolic Hitchin map $h_P$, using the following parabolic BNR correspondence.

\begin{thm}[{\cite[Theorem~4.9]{SWW22}}]\label{parabolic BNR}
 Assume that there exists a point in $\mathbf{H}_P$ whose corresponding spectral curve is integral. Then for a generic point $\alpha \in \mathbf{H}_P$, there is a bijective correspondence between the following two sets:
	\begin{itemize}\itemsep=0pt
		\item[$(1)$] parabolic Higgs bundles $(E,\phi)$ over $\mathbb{P}^1$ with $\operatorname{char}(E,\phi)=\alpha$;
		\item[$(2)$] line bundles on the normalization $\tilde{C}_{\alpha}$ of $C_{\alpha}$.
	\end{itemize}
\end{thm}

\begin{rem}\label{rem after par BNR}\quad
	\begin{itemize}\itemsep=0pt
		\item[(1)] The original BNR correspondence \cite{BNR89} states that for (non-parabolic) Higgs bundles on a smooth projective curve, there is a bijection between Higgs bundles and torsion-free rank-one sheaves on the associated integral spectral curve (see \cite[Proposition~3.6]{BNR89}). When the spectral curve is smooth, ``torsion-free rank $1$ sheaves'' is equivalent to ``line bundles''. In the parabolic setting, however, the spectral curves that appear are rarely smooth. Indeed, if some $n_i(x)\neq 1$, the corresponding spectral curve is singular.
		
		\item[(2)] In \cite{SWW22}, the base curve is assumed to have genus at least~$2$. Nevertheless, the arguments are essentially local, so all conclusions remain valid in the genus-zero case considered here.
	\end{itemize}
\end{rem}

\begin{cor}Assume that we have a point in $\mathbf{H}_P$ whose corresponding spectral curve is integral. Then for a generic point $\alpha \in \mathbf{H}_P$ the fiber $h_P^{-1}(\alpha)$ is isomorphic to a connected component of the Picard variety of $\tilde{C}_{\alpha}$. Moreover, we have
$\dim h_P^{-1}(\alpha) = \frac{1}{2} \dim \mathbf{Higgs}_P$.
\end{cor}

Since the parabolic Hitchin map $h_P$ is proper, we also have the following corollary.

\begin{cor}
 Assume that we have a point in $\mathbf{H}_P$ whose corresponding spectral curve is integral, then the image of the parabolic Hitchin map $h_P$ is $\mathbf{H}_P$.
\end{cor}

\section{Quiver varieties and the isomorphism theorem}\label{section 3}

A quiver is a finite oriented graph. Let $Q = (\operatorname{I}, \operatorname{E})$ be a quiver, where $\operatorname{I}$ is the set of vertices and $\operatorname{E}$ is the set of oriented edges. Given a dimension vector $\mathbf{v} = (v_i)_{i\in \operatorname{I}} \in \mathbb{Z}_{\ge 0}^{\operatorname{I}}$, a representation of $Q$ of dimension $\mathbf{v}$ consists of a family of vector spaces $\{V_i\}_{i\in \operatorname{I}}$ with $\dim V_i = v_i$ together with linear maps $\{\phi_{ij}\colon V_i\to V_j\}_{(i\to j)\in \operatorname{E}}$. All such representations are parametrized by the linear space
\[
\mathbf{R}= \operatorname{Rep}(Q,\mathbf{v})
= \bigoplus_{(i\to j)\in \operatorname{E}} \operatorname{Hom}(\mathbb{C}^{v_i},\mathbb{C}^{v_j}),
\]
on which the group $\operatorname{GL}(\mathbf{v}):=\prod_{i\in \operatorname{I}}\operatorname{GL}(v_i)$ acts by change of basis. The diagonal subgroup $\Delta\colon \mathbb{G}_m\hookrightarrow\operatorname{GL}(\mathbf{v})$ acts trivially. Two representations in $\operatorname{Rep}(Q,\mathbf{v})$ are isomorphic precisely when they lie in the same $\operatorname{GL}(\mathbf{v})$-orbit.

Now we define
$
\mathpzc{R}_0(\mathbf{v}) :=\mathbf{R}/\!/\operatorname{GL}(\mathbf{v})
=\operatorname{Spec} \mathbb{C}[\mathbf{R}]^{\operatorname{GL}(\mathbf{v})}$.
We also consider the following GIT quotient, following King~\cite{AK94}.

\begin{defn}\label{definition}
	Given a character $\chi\colon \operatorname{GL}(\mathbf{v})\rightarrow \mathbb{G}_m$, the GIT quotient of $\mathbf{R}$ by $\operatorname{GL}(\mathbf{v})$ with respect to $\chi$ is
	\[
	\mathpzc{R}_{\chi}(\mathbf{v})=\mathbf{R}/\!/_{\chi}\operatorname{GL}(\mathbf{v})
	=\operatorname{Proj}\biggl(\bigoplus_{n\geq 0}\mathbb{C}[\mathbf{R}]^{\operatorname{GL}(\mathbf{v}), \chi^{n}}\biggr),
	\]
	where
	\[
	\mathbb{C}[\mathbf{R}]^{\operatorname{GL}(\mathbf{v}), \chi^{n}}
	=\bigl\{f\in\mathbb{C}[\mathbf{R}]\mid
	f\bigl(g^{-1}x\bigr)=\chi(g)^{n}f(x)\text{ for all }g\in\operatorname{GL}(\mathbf{v}), x\in\mathbf{R}\bigr\}.
	\]
	The natural morphism $\mathpzc{R}_{\chi}(\mathbf{v})\rightarrow \mathpzc{R}_0(\mathbf{v})$ is projective.
\end{defn}

King also analyzed stability conditions in this setting.

\begin{lem}[{\cite[Proposition~2.5]{AK94}}]\label{criterion}
Let a reductive algebraic group $G$ act linearly on a vector space $\mathbf{R}$, and assume the kernel of the action is the central subgroup $\Delta$. Let $\chi\colon G\rightarrow \mathbb{G}_m$ be a~character.
	
A point $x\in\mathbf{R}$ is $\chi$-semistable if and only if $\chi\mid_\Delta=1$ and for every one-parameter subgroup $\lambda\colon \mathbb{G}_m\to G$ such that $\lim_{t\to0}\lambda(t)x$ exists, one has $\chi\circ \lambda(t)=t^{a}$ with $a\le0$.
	
	The point $x$ is stable if and only if the only one-parameter subgroup $\lambda$ of $G$ for which $\lim_{t\to0}\lambda(t)x$ exists and $\chi\circ \lambda(t)=1$ is contained in $\Delta$.
\end{lem}

In what follows we focus on the ``star-shaped'' quiver shown in Figure~\ref{Figure1} (the vertices are already labelled by the corresponding components of the dimension vector).

\begin{figure}[htbp]
$$
	\begin{tikzcd}
		& \gamma_1^1 \arrow[ddl] & \gamma^2_1 \arrow[l] & \cdots \arrow[l] & \gamma^{\sigma_1}_1 \arrow[l]\\
		& \gamma^1_2 \arrow[dl] & \gamma_2^2 \arrow[l] & \cdots \arrow[l] & \gamma^{\sigma_2}_2 \arrow[l]\\
		r & \cdots & \cdots & \cdots & \cdots \\
		& \gamma^1_n \arrow[ul] & \gamma^2_n \arrow[l] & \cdots \arrow[l] & \gamma^{\sigma_n}_n \arrow[l]
	\end{tikzcd}
$$
	\caption{The star-shaped quiver $Q$.}
	\label{Figure1}
\end{figure}

We now give an explicit description of $\mathpzc{R}_{\chi}(\mathbf{v})$ for the star-shaped quiver. We begin with a~lemma that analyses semistable points in a representation space with respect to a product of general linear groups. We record the following well-known result here.

\begin{lem}[{\cite[Example 10.25]{KA16}}]\label{A_n}
	Let $G=\operatorname{GL}\bigl(\gamma^1\bigr)\times\cdots\times\operatorname{GL}(\gamma^{\sigma})$ act on the space
	\[
	\mathbf{R}_A = \operatorname{Hom}\bigl(\mathbb{C}^{\gamma^{1}},\mathbb{C}^{r}\bigr)
	\oplus \operatorname{Hom}\bigl(\mathbb{C}^{\gamma^{2}},\mathbb{C}^{\gamma^{1}}\bigr)
	\oplus \cdots
	\oplus \operatorname{Hom}\bigl(\mathbb{C}^{\gamma^{\sigma}},\mathbb{C}^{\gamma^{\sigma-1}}\bigr)
	\]
	with $r>\gamma^1$ and $\gamma^i>\gamma^{i+1}$. Let $\chi\colon G\to\mathbb{G}_m$ be the character $(g_i)\mapsto \prod (\det g_i)^{a_i}$ with $a_i>0$.
	A point $f=(f_i)_{1\le i\le\sigma}$ is semistable if and only if each linear map $f_i$ has maximal rank, i.e., $\operatorname{rk} f_i = \gamma^i$. Moreover, semistability is equivalent to stability in this case.
	Consequently, $G$ acts freely on the semistable locus $\mathbf{R}_A^{ss}$, and the GIT quotient is isomorphic to the partial flag variety:
	\[
	\mathbf{R}_A/\!/_{\chi}G \cong \operatorname{Flag}(\mathbb{C}^{r},\vec{\gamma}),
	\]
	where $\vec{\gamma}=\bigl(\gamma^1,\dots,\gamma^{\sigma}\bigr)$.
\end{lem}

\begin{lem}\label{two quotient}	Let $G = G_1 \times G_2$ act linearly on a vector space $\mathbf{R}$, and let
	$\chi\colon G \rightarrow \mathbb{G}_m$ be a~character that factors as the product of characters
	$\chi_1$ of $G_1$ and $\chi_2$ of $G_2$. Then
	\[
	\mathbf{R}/\!/_{\chi}G \cong \bigl(\mathbf{R}/\!/_{\chi_1}G_1\bigr)/\!/_{\chi_2}G_2.
	\]
\end{lem}

\begin{proof}The equality $\mathbb{C}[\mathbf{R}]^{G, \chi^{n}} =
	\bigl(\mathbb{C}[\mathbf{R}]^{G_1, \chi_1^{n}}\bigr)^{G_2, \chi_2^{n}}$ follows directly from the definitions.
\end{proof}

We now establish an isomorphism between the quiver variety of the star-shaped quiver and the moduli space of semistable parabolic bundles on $\mathbb{P}^1$ constructed in Section~\ref{section 2}.
Recall the finite set $I=\{x_1,\dots,x_n\}$ and the parabolic type $\Sigma$.
Consider the quiver $Q$ depicted in Figure~\ref{Figure1} and choose a dimension vector $\mathbf{v}$ for $Q$ as indicated in the same figure.
Set $\gamma_i^j = \gamma^j(x_i)$ and $\sigma_i = \sigma_{x_i}-1$, where the numbers $\gamma^j(x_i)$ are those appearing in the construction of $\mathbf{M}_P$ (see Section~\ref{section 2}).

Define a character $\chi\colon \operatorname{GL}(\mathbf{v})\longrightarrow \mathbb{G}_m$ by
\[
\chi\bigl(g_0,\bigl(g_i^j\bigr)\bigr)= (\det g_0)^{-N}\prod_{i=1}^{n}\prod_{j=1}^{\sigma_i}\bigl(\det g_i^j\bigr)^{d_i^j},
\]
where $g_0\in\operatorname{GL}(r)$, $g_i^j\in\operatorname{GL}(\gamma_i^j)$, and
$d_i^j = d_j(x_i)=a_{j+1}(x_i)-a_j(x_i)$.
If necessary, we replace $\chi$ by a positive multiple so that
\[
N = \frac{1}{r}\sum_{i=1}^{n}\sum_{j=1}^{\sigma_i}\gamma_i^j d_i^j
\]
becomes an integer. This choice guarantees that $\chi(\Delta)=1$.

\begin{thm}\label{main1}
	Assume condition~\eqref{Condition} holds, and let $\mathbf{v}$ and $\chi$ be chosen as above.
	Then the moduli space $\mathbf{M}_P$ of rank $r$, degree $0$ semistable parabolic bundles of type $\Sigma$ on $\mathbb{P}^1$ is isomorphic to the quiver variety $\mathpzc{R}_{\chi}(\mathbf{v})$.
\end{thm}

\begin{proof}
	Fix an isomorphism $V\cong\mathbb{C}^{r}$. Write
	$G_1=\operatorname{GL}(r)$ and $G_2=\operatorname{GL}(\mathbf{v})/\operatorname{GL}(r)$.
	By Lemma~\ref{A_n}, the GIT quotient $\mathbf{R}/\!/_{\chi_2}G_2$ is isomorphic to the product of flag varieties $\mathbf{F}$ introduced in Section~\ref{section 2}.
	Applying Lemma~\ref{two quotient}, we obtain
$\mathbf{R}/\!/_{\chi}\operatorname{GL}(\mathbf{v}) \cong \mathbf{F}/\!/\operatorname{GL}(r)$,
	where the polarization on the right-hand side is the one described earlier.
	By Proposition~\ref{moduli of bundle}, the latter quotient is precisely $\mathbf{M}_P$ under condition~\eqref{Condition}.
\end{proof}

\begin{rem}\quad
 \begin{itemize}\itemsep=0pt
 \item Definition~\ref{definition} shows that replacing the character $\chi$ by a positive multiple $n\chi$ yields the same GIT quotient. By contrast, Example~\ref{key example} and Remark~\ref{key remark} illustrate that the weights of a parabolic bundle generally cannot be scaled without altering the moduli space. This distinction explains the necessity of imposing condition~\eqref{Condition} in our setting.
 \item This description is closely related to Soibelman's stack-theoretic description via squid algebras \cite{Soi14,Soi16}. In the degree-zero homologically trivial case, choosing a basis of global sections and trivializations of the flag subspaces gives the corresponding framed quiver representation. Passing to the quotient by the change of basis recovers the GIT quotient of the product of flag varieties described above.
 \end{itemize}
\end{rem}

We now show how the moduli space of homologically trivial semistable parabolic Higgs bundles on $\mathbb{P}^1$ can also be realized as a quiver variety.

First we recall the notion of the doubled quiver. For a quiver $Q=(\operatorname{I},\operatorname{E})$, its double $\overline{Q}$ is defined as $\overline{Q}=(\operatorname{I},\,\operatorname{E}\cup \operatorname{E}^{\mathrm{op}})$, where $\operatorname{E}^{\mathrm{op}}$ consists of an arrow $j\to i$ for every arrow $i\to j$ in $\operatorname{E}$. The doubled quiver associated with the star-shaped quiver of Figure~\ref{Figure1} is displayed in Figure~\ref{Figure2}.

\begin{figure}[htbp]
$$
	\begin{tikzcd}
		& \gamma_1^1 \arrow[ddl] \arrow[r, shift left=.7ex]
		& \gamma_1^2 \arrow[l] \arrow[r, shift left=.7ex]
		& \cdots \arrow[l] \arrow[r, shift left=.7ex]
		& \gamma^{\sigma_1}_1 \arrow[l]\\
		& \gamma^1_2 \arrow[dl] \arrow[r, shift left=.7ex]
		& \gamma_2^2 \arrow[l] \arrow[r, shift left=.7ex]
		& \cdots \arrow[l] \arrow[r, shift left=.7ex]
		& \gamma^{\sigma_2}_2 \arrow[l]\\
		r \arrow[uur, shift left=.7ex] \arrow[ur, shift left=.7ex] \arrow[dr, shift left=.7ex]
		& \cdots & \cdots & \cdots & \cdots \\
		& \gamma^1_n \arrow[ul] \arrow[r, shift left=.7ex]
		& \gamma^2_n \arrow[l] \arrow[r, shift left=.7ex]
		& \cdots \arrow[l] \arrow[r, shift left=.7ex]
		& \gamma^{\sigma_n}_n \arrow[l]
	\end{tikzcd}
$$
	\caption{The doubled quiver $\overline{Q}$.}
	\label{Figure2}
\end{figure}

The representation space $\operatorname{Rep}\bigl(\overline{Q},\mathbf{v}\bigr)$ of $\overline{Q}$ is canonically identified with the cotangent bundle of $\mathbf{R}$; consequently we obtain a moment map
$
\mu\colon   \operatorname{Rep}\bigl(\overline{Q},\mathbf{v}\bigr)=T^{*}\mathbf{R}\longrightarrow \mathfrak{g}_{\mathbf{v}}$,
defined by $(f_{ij},g_{ji})\mapsto \sum (f_{ij}\circ g_{ji}-g_{ji}\circ f_{ij})$, where $\mathfrak{g}_{\mathbf{v}}\cong\mathfrak{g}_{\mathbf{v}}^{*}$ is the Lie algebra of $\operatorname{GL}(\mathbf{v})$. The variety of interest is
$\mathfrak{M}_{\chi}(\mathbf{v})=\mu^{-1}(0)/\!/_{\chi}\operatorname{GL}(\mathbf{v})$.
Here $\chi$ is also allowed to be trivial, in which case, we have an affine variety $\mathfrak{M}_0(\mathbf{v})$. We have the following natural projective morphism:
$\mathfrak{M}_{\chi}(\mathbf{v})\longrightarrow \mathfrak{M}_0(\mathbf{v})$.

We now give an explicit description of $\mathfrak{M}_{\chi}(\mathbf{v})$ when $\overline{Q}$ and $\mathbf{v}$ are as in Figure~\ref{Figure2}. First we analyse the action of $G_{2}= \operatorname{GL}(\mathbf{v})/\operatorname{GL}(r)$ on $\mu^{-1}(0)$.

\begin{lem}\label{A_n^*}
	Let $G$, $\mathbf{R}_{A}$, $\chi$ and $\vec{\gamma}$ be as in Lemma~{\rm \ref{A_n}}.
	Consider the moment map $\tilde{\mu}\colon T^{*}\mathbf{R}_{A}\to\mathfrak{g}$ and the induced action of~$G$ on $\tilde{\mu}^{-1}(0)$. Then
$\tilde{\mu}^{-1}(0)/\!/_{\chi}G \cong T^{*}\!\operatorname{Flag}(\mathbb{C}^{r},\vec{\gamma})$.
\end{lem}

\begin{proof}
	A proof can be found in \cite[Theorem~10.43]{KA16}.
\end{proof}

Decompose $\operatorname{GL}(\mathbf{v})$ as $G_{1}\times G_{2}$ with $G_{1}=\operatorname{GL}(r)$. The moment map $\mu$ then decomposes accordingly, giving a commutative diagram
\begin{equation*}
	\xymatrix{
		\text{T}^*\mathbf{R} \ar[rd]^{\mu} \ar@{->}@/_{1.0pc}/ [rdd]_{\mu_1} \ar@{->}@/^{1.0pc}/ [rrd]^{\mu_2} & & \\
		& \mathfrak{g}_{\mathbf{v}}=\mathfrak{g}_1\oplus \mathfrak{g}_2 \ar[d] \ar[r] & \mathfrak{g}_2,\\
		& \mathfrak{g}_1 & }
\end{equation*}
where the vertical and horizontal maps are the natural projections.

Again using the isomorphism $V\cong\mathbb{C}^{r}$ fixed in Theorem~\ref{main1}, Lemma~\ref{A_n^*} implies that
\[
\mu_{2}^{-1}(0)/\!/_{\chi_{2}}G_{2}\cong T^{*}\mathbf{F}.
\]
The map $\mu_{1}\colon \mu_{2}^{-1}(0)\to\mathfrak{g}_{1}$ is $G_{2}$-invariant, hence it descends to a morphism
\[
\mu_{Q}\colon \ \mu_{2}^{-1}(0)/\!/_{\chi_{2}}G_{2}\longrightarrow \mathfrak{g}_{1}.
\]

\begin{thm}\label{main2}
	Under the same hypotheses as in Theorem~{\rm \ref{main1}}, there is an isomorphism
	\[
	\Psi\colon \ \mathbf{Higgs}^{\circ}_P \longrightarrow \mathfrak{M}_{\chi}(\mathbf{v}),
	\]
	where $\mathbf{Higgs}^{\circ}_P$ denotes the moduli space of rank~$r$, degree~$0$
	homologically trivial semistable parabolic Higgs bundles of type~$\Sigma$ on~$\mathbb{P}^1$.
\end{thm}

\begin{proof}
Fix the isomorphism $V\cong\mathbb{C}^{r}$. From the construction in Section~\ref{section 2}
	we have the morphism $\mu_P\colon T^{*}\mathbf{F}\to\mathscr{H}\!{\rm om}(V,V)$ introduced before
	Proposition~\ref{HiggsP}. By Lemma~\ref{A_n^*}, there is an identification
	$T^{*}\mathbf{F}\cong\mu_{2}^{-1}(0)/\!/_{\chi_{2}}G_{2}$. Moreover, under the natural
	identifications $\mathscr{H}\!{\rm om}(V,V)\cong\mathfrak{g}_{1}$, the map $\mu_P$ corresponds
	precisely to the descended map $\mu_{Q}$. Hence we obtain a commutative diagram
	\[
	\begin{tikzcd}[row sep=large]
		T^{*}\mathbf{F} \ar[r,"\mu_P"] \ar[d,"\cong"'] & \operatorname{Hom}(V,V) \ar[d,"\cong"] \\[4mm]
		\mu_{2}^{-1}(0)/\!/_{\chi_{2}}G_{2} \ar[r,"\mu_{Q}"'] & \mathfrak{g}_{1}.
	\end{tikzcd}
	\]
	Since $\mu^{-1}(0)=\mu_{1}^{-1}(0)\cap\mu_{2}^{-1}(0)$, the theorem follows as in the proof of Theorem~\ref{main1}.
\end{proof}

\begin{rem}\label{sym iso}
The Higgs version is the cotangent analogue of the preceding construction. It corresponds to imposing the moment-map equation for the doubled star-shaped quiver, parallel to the cotangent/moment-map picture for squid representations in Soibelman's framework.
\end{rem}

We now describe the isomorphism $\Psi$ explicitly at the level of points.

Given a homologically trivial semistable parabolic Higgs bundle $(E,\phi)$,
we have $E\cong\mathcal{O}_{\mathbb{P}^1}^{\oplus r}$ and may identify
$H^0\bigl(\mathbb{P}^1,E\bigr)\cong V\cong\mathbb{C}^{r}$.
By the exact sequence~\eqref{exseq}, the Higgs field $\phi$ is equivalent to a collection of
$n$ linear maps $\phi_i\in\operatorname{Hom}_{x_i}^{\mathrm{s.f.}}(\mathbb{C}^{r},\mathbb{C}^{r})$
satisfying $\sum_{i=1}^{n}\phi_i=0$.
At each marked point $x_i$ the parabolic filtration on $E|_{x_i}$ induces a filtration
$F^{\bullet}(\mathbb{C}^{r})$. Restricting $\phi_i$ yields maps
$F^{j}(\mathbb{C}^{r})\to F^{j+1}(\mathbb{C}^{r})$; together with the inclusions
$F^{j+1}(\mathbb{C}^{r})\hookrightarrow F^{j}(\mathbb{C}^{r})$ these data assemble into a
representation $\mathsf{V}$ of the doubled quiver $\overline{Q}$.
Because $\sum_i\phi_i=0$, the representation lies in $\mu^{-1}(0)$; the semistability of
$(E,\phi)$ guarantees that $\mathsf{V}$ is $\chi$-semistable.
Thus $\mathsf{V}$ defines a point of $\mathfrak{M}_{\chi}(\mathbf{v})$.

Conversely, let $\mathsf{V}\in\mu^{-1}(0)$ be a $\chi$-semistable representation.
For each leg $i=1,\dots,n$, we have maps
\[
\mathbb{C}^{r}\xrightarrow{\;f^1_{\,i}\;}\mathbb{C}^{\gamma^1_{\,i}}
\xrightarrow{\;g^1_{\,i}\;}\mathbb{C}^{r},
\qquad \sum_{i=1}^n g^1_{\,i}\circ f^1_{\,i}=0 .
\]
Semistability of $\mathsf{V}$ implies that each map $g^1_{\,i}$ is injective, thereby producing
$n$ filtrations on $\mathbb{C}^{r}$. The condition $\mathsf{V}\in\mu^{-1}(0)$ forces
the composite $g^1_{\,i}\circ f^1_{\,i}$ to lower the corresponding filtration by one step.
Using the exact sequence~\eqref{exseq} again, this collection of maps determines a parabolic
Higgs field on $\mathcal{O}_{\mathbb{P}^1}^{\oplus r}$, and hence a semistable parabolic
Higgs bundle.

\section{The nilpotent case of the additive Deligne--Simpson problem}\label{section 5}

The additive Deligne--Simpson problem can be stated as follows.

\begin{problem}\label{prob1}
 Given $n$ conjugacy classes $\mathfrak{c}_i\subset\mathfrak{gl}_r$, can one find matrices $A_i\in\mathfrak{c}_i$ such that $\sum_{i=1}^n A_i = 0$?
\end{problem}
A solution $\{A_i\}$ is called irreducible if the matrices have no common proper nonzero invariant subspace.

One motivation for this problem comes from linear systems of differential equations on the Riemann sphere. Consider a meromorphic matrix function $A(z)$ on $\mathbb{CP}^1$ with poles at $x_1,\dots,x_n$. The system ${\rm d}X/{\rm d}z = A(z)X$ is Fuchsian if all poles are logarithmic; it can then be written as
\[
\frac{{\rm d}X}{{\rm d}z}= \left(\sum_{i=1}^n\frac{A_i}{z-x_i}\right)X,
\qquad A_i\in\mathfrak{gl}_r,\qquad \sum_{i=1}^n A_i = 0 .
\]
Thus the existence of a Fuchsian system with residues in prescribed conjugacy classes is equivalent to the solvability of a corresponding Deligne--Simpson problem. For a thorough discussion, we refer to \cite{Kos04}.

The additive Deligne--Simpson problem has been studied extensively. Kostov \cite{Kos96,Kos03} obtained existence criteria for several special cases, including the nilpotent case. A complete criterion for arbitrary conjugacy classes was later given by Crawley-Boevey~\cite{CB03} via quiver representations. In~\cite{Soi16}, Soibelman used the ``very good'' property of certain moduli stacks of parabolic bundles on~$\mathbb{P}^1$ to study the solution space of the problem.

Motivated by Crawley-Boevey's result and the isomorphism we have established in Theorem~\ref{main2}, in this section we attack the nilpotent case of the Deligne--Simpson problem from a geometric viewpoint, using parabolic Higgs bundles on $\mathbb{P}^1$ and the associated spectral curve.

In this paper, we consider the following problem.

\begin{problem}\label{prob2}
 Assume the $n$ conjugacy classes $\mathbf{O}_i\subset\mathfrak{gl}_r$ are nilpotent orbits. Can one find an irreducible solution of Problem~\ref{prob1}?
\end{problem}

Fix $N_i\in \mathbf{O}_i$, let $\gamma_i^j = \operatorname{rk}(N_i^{\,j})$. Then the sequence $\{\gamma_i^j\}$ also determines $\mathbf{O}_i$ uniquely.
Consider rank $r$, degree $0$ (homologically trivial) parabolic Higgs bundles on $\mathbb{P}^1$ of type
\begin{equation}\label{types associated to conjugacy classes}
 \Sigma=\{I,K,\{n_j(x_i)\},\{a_j(x_i)\}\},
\end{equation}
where the weights satisfy condition~\eqref{Condition} and where{\samepage
\[
\gamma^j(x_i)=\sum_{k=j+1}^{\sigma_{x_i}}n_k(x_i)=\gamma_i^j .
\]
Denote the corresponding moduli space by $\mathbf{Higgs}_P$ ($\mathbf{Higgs}^{\circ}_P$).}
The nilpotent orbit is also~determined by the Jordan type of $N_i$, which is given by
\smash{$\bigl[\mu_i^k:=\mu_k(x_i)\bigr]_{k\geq 1}$}
as in~\eqref{muj}. So \smash{$\bigl\{\mu_i^j\bigr\}$} also determines the nilpotent orbit $\mathbf{O}_i$ uniquely. For each $x_i$ and $\ell$, we define
\smash{$s_i^{\ell}=\sum_{k\leq \ell}\mu_i^k$}.
We also define
\smash{$J_i=\bigl\{s_i^\ell\mid \mu_i^\ell>\mu_i^{\ell+1}\bigr\}$}.

Let $(E,\phi)$ be a homologically trivial parabolic Higgs bundle corresponding to a closed point of $\mathbf{Higgs}^{\circ}_P$. Choose the local basis ${\rm d}z/(z-x_i)$ for $\omega_{\mathbb{P}^1}(D_I)|_{x_i}$ and fix an isomorphism $\operatorname{H}^0\bigl(\mathbb{P}^1,E\bigr)\cong \mathbb{C}^{\oplus r}$.

By the exact sequence~\eqref{exseq} in Section~\ref{section 2}, the pair $(E,\phi)$ is equivalent to a collection of~$n$ linear maps $A_i\colon \mathbb{C}^{\oplus r}\to \mathbb{C}^{\oplus r}$ that lower the corresponding filtration by one step on $\mathbb{C}^{\oplus r}$ induced by the parabolic structure of $E$ and satisfy $\sum_i A_i = 0$.

At this stage, however, we cannot yet control the conjugacy class of each $A_i$; for instance, it is not clear whether $\operatorname{rk} A_i = \gamma_i^1$.

We therefore study the parabolic Hitchin map $h_P\colon \mathbf{Higgs}_P\rightarrow \mathbf{H}_P$.
Recall that the Hitchin base corresponding to the type $\Sigma$ as in~\eqref{types associated to conjugacy classes} is
\[
\mathbf{H}_P:=\prod_{j=1}^{r} \operatorname{H}^{0}\left(\mathbb{P}^1,
\omega_{\mathbb{P}^1}^{\otimes j}\otimes\mathcal{O}_{\mathbb{P}^1}
\left(\sum_{i=1}^n(j-\varepsilon_j(x_i))x_i\right)\right),
\]
where the numbers $\varepsilon_{j}(x_i)$ are defined as in Section~\ref{section 2}.

The following lemma shows that the order of coefficients of the parabolic Hitchin image at $x_i$ can detect the Jordan types of $\mathrm{Res}_{x_i}(\phi)$, namely the Jordan types of $A_i$.

\begin{lem}\label{thetype}
 We fix local trivializations ${\rm d}z/(z-x_i)$ at $x_i$ for $\omega_{\mathbb{P}^1}(D_I)$ as before. For a parabolic Higgs bundle $(E,\phi)$, assume its image under $h_P$ consists of the sections
$(\alpha_1, \dots, \alpha_r)\in \mathbf{H}_P$.
If, for every $i$ and every $j\in J_i$, the order of vanishing of $\alpha_j$ at $x_i$ equals $\varepsilon_j(x_i)$, then
$\mathrm{Res}_{x_i}(\phi)\in\mathbf{O}_i$
 for every~$i$.
\end{lem}

\begin{proof}It suffices to prove the assertion for one fixed $i$. Put $A_i=\mathrm{Res}_{x_i}(\phi)$, and let $P_i\subset G=\operatorname{GL}(r)$ be the parabolic subgroup corresponding to the prescribed flag type at $x_i$. Since $A_i$ lowers this flag by one step, it belongs to the image of the moment map
$T^*(G/P_i)\longrightarrow\mathfrak{gl}_r$,
 which is $\overline{\mathbf{O}}_i$.

 We recall explicitly the local vanishing estimate underlying Proposition~\ref{Hitchin base}; see also \cite[Theorem~3.4]{SWW22}. For every flag at $x_i$ that is lowered by $A_i$ by one step, the coefficients $\alpha_j$ satisfy the vanishing-order bounds computed from that flag by~\eqref{eq:epsilonj}. In particular, the prescribed flag gives
$\operatorname{ord}_{x_i}(\alpha_j)\geq\varepsilon_j(x_i)$.
The coefficients $\alpha_j$ depend only on the Higgs field $\phi$ and remain unchanged when another compatible flag is considered. Only the bounds computed from the flag may change. This estimate is local and does not require stability for the auxiliary flag.

 Suppose that $A_i$ lies in a proper nilpotent orbit $\mathbf{O}\subset\overline{\mathbf{O}}_i$. Write its Jordan type as \smash{$\bigl[\mu^k\bigr]_{k\geq1}$}, while \smash{$\bigl[\mu_i^k\bigr]_{k\geq1}$} is the prescribed Jordan type. Then~$A_i$ lowers the filtration
 \[
 E|_{x_i}\supset\operatorname{Im}A_i
 \supset\operatorname{Im}A_i^2\supset\cdots\supset0,
 \]
 whose graded dimensions have dual partition \smash{$\bigl[\mu^k\bigr]_{k\geq1}$}. Denote the bounds associated with this filtration by $\widetilde{\varepsilon}_j(x_i)$; thus, following~\eqref{eq:epsilonj},
 \[
 \widetilde{\varepsilon}_j(x_i)=q
 \ \Longleftrightarrow\
 \sum_{k\leq q-1}\mu^k<j\leq\sum_{k\leq q}\mu^k.
 \]
 We apply the local vanishing estimate to the filtration above. Notice that the characteristic coefficients remain unchanged, whereas the resulting bounds are now $\widetilde{\varepsilon}_j(x_i)$. Hence
\smash{$\operatorname{ord}_{x_i}(\alpha_j) \geq\widetilde{\varepsilon}_j(x_i)$}.

 By the dominance order on nilpotent orbits,
 \begin{equation}\label{smaller}
 \sum_{k\leq m}\mu^k\leq\sum_{k\leq m}\mu_i^k
 \end{equation}
 for every $m$, with strict inequality for at least one $m$. Here a corner of the partition $[\mu_i^k]$ means an index $\ell$ such that $\mu_i^\ell>\mu_i^{\ell+1}$. Strict inequality in~\eqref{smaller} must occur for at least one such index. Indeed, if equality held at every corner, then on each block of equal parts of \smash{$\bigl[\mu_i^k\bigr]$} its partial sums would be linear, while those of \smash{$\bigl[\mu^k\bigr]$} are concave. Equality at the endpoints and the dominance inequalities would force equality throughout each block, hence equality of the two partitions, a~contradiction.

 It follows that we have an index $j=s_i^\ell\in J_i$ such that
 \[
 \sum_{k\leq\ell}\mu^k<j=\sum_{k\leq\ell}\mu_i^k.
 \]
 By~\eqref{eq:epsilonj} and the definition of $\widetilde{\varepsilon}_j(x_i)$, we obtain
 \[
 \operatorname{ord}_{x_i}(\alpha_j)
 \geq\widetilde{\varepsilon}_j(x_i)
 \geq\ell+1
 >\ell=\varepsilon_j(x_i),
 \]
 contradicting the assumed minimal vanishing at $j\in J_i$. Hence $A_i\in\mathbf{O}_i$.
\end{proof}

Now take a homologically trivial parabolic Higgs bundle $(E,\phi)$ whose
Hitchin image has the minimal vanishing orders required in Lemma~\ref{thetype}.
By that lemma and~\eqref{exseq}, it gives a~solution of Problem~\ref{prob1} in the
prescribed nilpotent orbits.
Since~$E$ is homologically trivial, $E\simeq\mathcal O_{\mathbb P^1}
\otimes H^0\bigl(\mathbb P^1,E\bigr)$. Thus a~common proper nonzero invariant
subspace~$W$ of the residues determines a~proper nonzero parabolic Higgs
subbundle $\mathcal O_{\mathbb P^1}\otimes W$.
So if the solution $\{A_i\}$ admits a common proper nonzero invariant subspace~$W$, then $\mathcal{O}_{\mathbb{P}^1}\otimes W$ is a proper parabolic Higgs subbundle of~$(E,\phi)$.

Then the following lemma is immediate.

\begin{lem}\label{inte}
 For a homologically trivial parabolic Higgs bundle $(E,\phi)$, if
its Hitchin image satisfies the vanishing-order conditions of Lemma~{\rm \ref{thetype}}
 and the spectral curve associated to $h_P ((E,\phi))$ is integral, then $(E,\phi)$ gives an irreducible solution of Problem~{\rm \ref{prob2}}.
\end{lem}

\begin{proof}
 By Lemma~\ref{thetype}, the matrices $A_i$ associated with $(E,\phi)$ lie in the nilpotent orbits~$\mathbf{O}_i$.
Suppose that the matrices $A_i$ have a common proper nonzero invariant subspace~$W$.
Then $\mathcal O_{\mathbb P^1}\otimes W$ is a proper nonzero Higgs subbundle of~$(E,\phi)$.
The characteristic polynomial of~$\phi$ therefore factors over the function field of~$\mathbb P^1$, contradicting
the integrality of the spectral curve.
\end{proof}

For such sections $\{\alpha_j\}$ to exist in $\mathbf{H}_P$, by Lemma~\ref{thetype}, we may require the inequalities
\[
-2j+\sum_{i=1}^n (j-\varepsilon_j(x_i) )\ge 0
\]
for every $j$ such that $j=s_i^\ell$ and $\mu_i^\ell\neq \mu_i^{\ell+1}$. Such inequalities can be obtained by a simpler condition. Recall that $\gamma_i^1=r-\varepsilon_r(x_i)$.

\begin{lem}\label{one condition}
 If
 \begin{equation}\label{DS condition}
	 2r<\sum_{i=1}^n\gamma_i^{1},
 \end{equation}
 or if $2r=\sum_{i=1}^n\gamma_i^{1}$ and
 $\mathrm{gcd}\bigl(r,\gamma_1^1,\dots,\gamma_n^1\bigr)=1$, then
 \[
 -2j+\sum_{i=1}^n (j-\varepsilon_j(x_i) )\ge 0
 \]
 for every $j<r$ such that there exists an $i_0$ for which
 $j\in J_{i_0}$.
\end{lem}

\begin{proof}
 Recall the definition of $\varepsilon_j(x_i)$ in~\eqref{eq:epsilonj}:
 \[
 \varepsilon_j(x_i)=q
	\ \Longleftrightarrow\
	\sum_{t\le q-1}\mu_t(x_i)<j\le \sum_{t\le q}\mu_t(x_i).
 \]
 Since $\sum_{t\le q-1}\mu_t(x_i)<j$, and since
 $\{\mu_t(x_i)\}$ is a non-increasing sequence with
 $\varepsilon_r(x_i)$ parts, we~have
 \begin{equation*}
 j\geq \sum_{t\le q-1}\mu_t(x_i)+1
 \geq (q-1)\left(\frac{r}{\varepsilon_r(x_i)}\right)+1.
 \end{equation*}
 It follows that
 \begin{equation}\label{first estimate}
 \varepsilon_j(x_i)\leq
 1+\left\lfloor
 \frac{(j-1)\varepsilon_r(x_i)}{r}
 \right\rfloor.
 \end{equation}

 We now obtain a sharper estimate for the index $i_0$ satisfying
 $j=s_{i_0}^\ell<r$ and
 $\mu_{i_0}^\ell\neq \mu_{i_0}^{\ell+1}$. In this case,
$
 \varepsilon_j(x_{i_0})=\ell$,
$j=s_{i_0}^\ell\geq \ell\mu_{i_0}^\ell$,
$\mu_{i_0}^{\ell+1}\leq \mu_{i_0}^\ell-1$.
 Therefore,
 \begin{equation}\label{r-j}
 r-j
 =\sum_{t=\ell+1}^{\varepsilon_r(x_{i_0})}\mu_{i_0}^t
 \leq
 (\varepsilon_r(x_{i_0})-\ell )
 \bigl(\mu_{i_0}^\ell-1\bigr).
 \end{equation}
 Consequently,
 \begin{align*}
 \varepsilon_r(x_{i_0})(j-1)+1-\ell r
 &=(\varepsilon_r(x_{i_0})-\ell)j
 -\varepsilon_r(x_{i_0})+1-\ell(r-j) \\
 &\geq
 (\varepsilon_r(x_{i_0})-\ell)
 \bigl(j-\ell\bigl(\mu_{i_0}^\ell-1\bigr)\bigr)
 -\varepsilon_r(x_{i_0})+1 \\
 &\geq
 (\varepsilon_r(x_{i_0})-\ell)\ell
 -\varepsilon_r(x_{i_0})+1 \\
 &=(\ell-1)
 (\varepsilon_r(x_{i_0})-\ell-1 )
  \geq 0.
 \end{align*}
 Here, the first inequality follows from~\eqref{r-j}, while the
 second follows from
 $j=s_{i_0}^\ell\geq \ell\mu_{i_0}^\ell$. The final inequality
 follows from the assumption $j<r$, which implies that
 $\ell<\varepsilon_r(x_{i_0})$. Thus,
 \begin{equation}\label{second estimate}
 \varepsilon_r(x_{i_0})(j-1)\geq \ell r-1.
 \end{equation}

 Set
 \begin{gather*}
 T_j
 =
 \sum_{i=1}^n
 \left(
 \frac{(j-1)\varepsilon_r(x_i)}{r}
 -
 \left\lfloor
 \frac{(j-1)\varepsilon_r(x_i)}{r}
 \right\rfloor
 \right)
 \geq 0,
\\
 F=\sum_{i=1}^n\varepsilon_r(x_i)-(n-2)r.
 \end{gather*}
 We then compute
 \begin{align}
 \sum_{i=1}^n \varepsilon_j(x_i)-(n-2)j
 &\leq
 n-1+
 \sum_{i\neq i_0}
 \left\lfloor
 \frac{(j-1)\varepsilon_r(x_i)}{r}
 \right\rfloor
 +\ell-(n-2)j
 \notag\\
 &=
 n-1+
 \sum_{i=1}^n
 \frac{(j-1)\varepsilon_r(x_i)}{r}
 -T_j+\ell
 -
 \left\lfloor
 \frac{(j-1)\varepsilon_r(x_{i_0})}{r}
 \right\rfloor
 -(n-2)j
 \notag\\
 &=
 1+\frac{(j-1)F}{r}-T_j+\ell
 -
 \left\lfloor
 \frac{(j-1)\varepsilon_r(x_{i_0})}{r}
 \right\rfloor.
 \notag
 \end{align}

 We first consider the case in which~\eqref{DS condition} holds,
 so that $F<0$. By~\eqref{first estimate}, we have
 \[
 \left\lfloor
 \frac{(j-1)\varepsilon_r(x_{i_0})}{r}
 \right\rfloor
 \geq \ell-1.
 \]
 Since $T_j\geq 0$, if $\bigl\lfloor \frac{(j-1)\varepsilon_r(x_{i_0})}{r} \bigr\rfloor \geq \ell$, then $\sum_{i=1}^n\varepsilon_j(x_i)-(n-2)j\leq 0$. It therefore remains to consider the case
 \[
 \left\lfloor
 \frac{(j-1)\varepsilon_r(x_{i_0})}{r}
 \right\rfloor
 =\ell-1.
 \]
 In this case,~\eqref{second estimate} implies that $(j-1)\varepsilon_r(x_{i_0})=\ell r-1$. Moreover, $\mathrm{gcd}(r,j-1)=1$. We claim that
 \[
 T_j-\frac{(j-1)F}{r}\geq 2.
 \]
 Since $T_j-\frac{(j-1)F}{r}$ is a positive integer, suppose, to the contrary, that $ T_j-\frac{(j-1)F}{r}=1$. Then
 \begin{align*}
 1
 =T_j-\frac{(j-1)F}{r}
  =
 \sum_{i\neq i_0}
 \left(
 \frac{(j-1)\varepsilon_r(x_i)}{r}
 -
 \left\lfloor
 \frac{(j-1)\varepsilon_r(x_i)}{r}
 \right\rfloor
 \right)
 +1-\frac{1}{r}-\frac{(j-1)F}{r}.
 \end{align*}
 Since $F<0$, this forces
 \[
 F=-1,\qquad j-1=1, \qquad r\mid (j-1)\varepsilon_r(x_i),
 \]
 for every $i\neq i_0$. Since $\mathrm{gcd}(r,j-1)=1$, it follows
 that $\varepsilon_r(x_i)=r$ for every $i\neq i_0$. This gives
 \[
 \sum_{i=1}^n\varepsilon_r(x_i)-(n-2)r
 =
 (n-1)r+\varepsilon_r(x_{i_0})-(n-2)r
 >0,
 \]
 contradicting the assumption.

 We now consider the second case, in which $F=0$ and $\mathrm{gcd}\bigl(r,\gamma_1^1,\dots,\gamma_n^1\bigr)=1$. To prove that%
 \[
 \sum_{i=1}^n\varepsilon_j(x_i)-(n-2)j\leq 0,
 \]
 it is enough to show that
 \[
 1-T_j+\ell
 -
 \left\lfloor
 \frac{(j-1)\varepsilon_r(x_{i_0})}{r}
 \right\rfloor
 \leq 0.
 \]
 If $\bigl\lfloor \frac{(j-1)\varepsilon_r(x_{i_0})}{r} \bigr\rfloor \geq \ell+1$, then the desired inequality holds immediately. If
 \[
 \left\lfloor
 \frac{(j-1)\varepsilon_r(x_{i_0})}{r}
 \right\rfloor
 =\ell,
 \]
 it remains to show that $T_j>0$. Since $T_j$ is an integer,
 suppose that $T_j=0$. Then we have $ r\mid (j-1)\varepsilon_r(x_i)$ for every $i$. Since $j-1<r$, this implies that{\samepage
 \[
 \mathrm{gcd}
 \bigl(r,\varepsilon_r(x_1),\dots,\varepsilon_r(x_n)\bigr)
 \neq 1,
 \]
 contradicting the assumption.}

 Finally, suppose that
 \[
 \left\lfloor
 \frac{(j-1)\varepsilon_r(x_{i_0})}{r}
 \right\rfloor
 =\ell-1.
 \]
 In this case, we claim that $T_j\geq 2$. The case $T_j=0$ is
 ruled out by the argument above. Suppose now that $T_j=1$. By
 an argument similar to that used in the first case, there exists
 an index $i_1\neq i_0$ such that $\varepsilon_r(x_i)=r$ for every $i\neq i_0,i_1$. It follows that
 \[
 \sum_{i=1}^n\varepsilon_r(x_i)-(n-2)r
 =
 (n-2)r+\varepsilon_r(x_{i_0})
 +\varepsilon_r(x_{i_1})-(n-2)r
 >0,
 \]
 again contradicting the assumption.
\end{proof}

By Lemma~\ref{inte}, we also need to decide whether there exists
an $\alpha\in\mathbf{H}_P$ whose associated spectral curve is integral.
A criterion from \cite[Remark~3.1]{BNR89} states that if one can
choose a nonzero section
\[
\alpha_r\in \operatorname{H}^{0}\left(\mathbb{P}^1,
\omega_{\mathbb{P}^1}^{\otimes r}\otimes\mathcal{O}_{\mathbb{P}^1}
\left(\sum_{i=1}^n(r-\varepsilon_r(x_i))x_i\right)\right)
\]
whose zero divisor, computed by viewing $\alpha_r$ as a section of
$\omega_{\mathbb{P}^1}(D_I)^{\otimes r}$, is not of the form~$mD$
for any integer $m>1$ dividing $r$, then the spectral curve
associated with $\alpha_0=(0,\ldots,0,\alpha_r)$, defined by
$\lambda^r+\alpha_r=0$, is integral.
Consequently, we obtain the following lemma.

\begin{lem}\label{integral lemma}
 Assume that either $2r<\sum_{i=1}^n\gamma_i^1$, or
 $2r=\sum_{i=1}^n\gamma_i^1$ and
 $\mathrm{gcd}\bigl(r,\gamma_1^1,\ldots,\gamma_n^1\bigr)=1$.
 Then there exists an element $\alpha_0\in\mathbf{H}_P$
 for which the associated spectral curve $C_{\alpha_0}$ is integral.
\end{lem}

\begin{proof}
 If $\sum_i\gamma_i^1>2r$, the relevant line bundle has positive
 degree, and a general choice of $\alpha_r$ has a~simple zero
 away from the marked points; its divisor therefore cannot
 be divisible by an integer $m>1$.
 If $\sum_i\gamma_i^1=2r$, the divisor is
$\sum_i\bigl(r-\gamma_i^1\bigr)x_i$.
 If it were divisible by an integer $m>1$ with $m\mid r$,
 then $m$ would divide every $\gamma_i^1$,
 contradicting the gcd hypothesis.
\end{proof}

\begin{thm}\label{DSP}
	If the conditions of Lemma~{\rm \ref{integral lemma}} are satisfied, then the nilpotent case of the Deligne--Simpson problem admits irreducible solutions. Moreover, these solutions can be constructed explicitly.
\end{thm}

\begin{proof}
The conditions that $C_\alpha$ be integral and that $\operatorname{ord}_{x_i}(\alpha_j)=\varepsilon_j(x_i)$ for every $i$ and every $j\in J_i$ are both open. By Lemma~\ref{integral lemma}, there exists $\alpha_0\in\mathbf{H}_P$ whose associated spectral curve $C_{\alpha_0}$ is integral. Also by Lemma~\ref{one condition}, the conditions in Lemma~\ref{integral lemma} also ensure that we have $\{\alpha_j\}\in \mathbf{H}_P$ such that, for every $i$ and every $j\in J_i$, the vanishing order of $\alpha_j$ at $x_i$ is exactly $\varepsilon_j(x_i)$, here the $j=r$ case follows from our assumption. Hence we may choose $\alpha\in\mathbf{H}_P$ that simultaneously satisfies the hypotheses of Theorem~\ref{parabolic BNR} and the vanishing orders condition.

 Now by Lemma~\ref{inte}, we need to find a homologically trivial parabolic Higgs bundle $(E, \phi)$ so that its Hitchin image is $\alpha$.
	
	Let $\tilde{C}_{\alpha}$ be the normalization of $C_{\alpha}$ and $\tilde{\pi}_{\alpha}:\tilde{C}_{\alpha}\to\mathbb{P}^1$ the projection. By Theorem~\ref{parabolic BNR},
	\[
	h_P^{-1}(\alpha)\cong\bigl\{\text{line bundles }\mathcal{L}\text{ on }\tilde{C}_{\alpha}
	\mid \deg\tilde{\pi}_{\alpha*}\mathcal{L}=0\bigr\}.
	\]
	The locus of homologically trivial parabolic Higgs bundles corresponds to those $\mathcal{L}$ with
	$\tilde{\pi}_{\alpha*}\mathcal{L}\cong\mathcal{O}_{\mathbb{P}^1}^{\oplus r}$; by Lemma~\ref{line bundle} below, it is a non-empty affine open subset inside
	$h_P^{-1}(\alpha)$.
	
Summarizing, one can construct such solutions explicitly as follows:
\begin{enumerate}\itemsep=0pt
\item Start with the $n$ nilpotent conjugacy classes $\mathfrak{c}_i$ and form the corresponding Hitchin base~$\mathbf{H}_P$.
\item Select a generic $\alpha\in\mathbf{H}_P$ satisfying the conditions of Theorem~\ref{parabolic BNR}, fix local trivializations of $\omega_{\mathbb{P}^1}(D_I)$ and compute the normalization $\tilde{C}_{\alpha}$ of the spectral curve $C_{\alpha}$ with projection $\tilde{\pi}_{\alpha}\colon \tilde{C}_{\alpha}\to\mathbb{P}^1$.
		\item Choose a line bundle $\mathcal{L}$ on $\tilde{C}_{\alpha}$ such that $\tilde{\pi}_{\alpha*}\mathcal{L}\cong\mathcal{O}_{\mathbb{P}^1}^{\oplus r}$.
 \item Fix a basis of $\operatorname{H}^0(\tilde{C}_{\alpha}, \mathcal{L})$ and then compute the matrix of the Higgs field $\phi$ with respect to this basis.
 \item Evaluate the residues of the Higgs field at the marked points $x_i\in I$ to get an irreducible solution.\hfill\qed
	\end{enumerate}\renewcommand{\qed}{}
\end{proof}

\begin{lem}\label{line bundle}
Let $f\colon C\to \mathbb P^1$ be a finite morphism of degree $r$, where $C$
is a smooth projective curve of genus $g$. Then the set
$\bigl\{\mathcal L\in \operatorname{Pic}^{g-1+r}(C)
\mid f_*\mathcal L\simeq \mathcal O_{\mathbb P^1}^{\oplus r}\bigr\}$
is a non-empty affine open subset of $\operatorname{Pic}^{g-1+r}(C)$.
More precisely, after tensoring by $f^*\mathcal O_{\mathbb P^1}(-1)$, it
is identified with the complement of the theta divisor in
$\operatorname{Pic}^{g-1}(C)$.
\end{lem}

\begin{proof}
For a line bundle $\mathcal L\in \operatorname{Pic}^{g-1+r}(C)$, set $\mathcal M=\mathcal L\otimes f^*\mathcal O_{\mathbb P^1}(-1)$. Then $\deg \mathcal M=g-1$, and $f_*\mathcal L\simeq \mathcal O_{\mathbb P^1}^{\oplus r}$
is equivalent to
$\operatorname{H}^0(C,\mathcal M)=\operatorname{H}^1(C,\mathcal M)=0$.
For degree $g-1$ line bundles, this is equivalent to
$\mathcal M$ lying outside the theta divisor. The complement of the theta
divisor is a~non-empty affine open subset of $\operatorname{Pic}^{g-1}(C)$.
Tensoring back by $f^*\mathcal O_{\mathbb P^1}(1)$ gives the desired open
subset of $\operatorname{Pic}^{g-1+r}(C)$.
\end{proof}

\begin{rem}\quad
 \begin{itemize}\itemsep=0pt
 \item The inequality~\eqref{DS condition} is familiar from the classical approaches to the Deligne--Simpson problem. It is closely related to the defect in Katz's middle convolution algorithm, as discussed for example in Katz~\cite{Kat96}, Simpson~\cite{Sim09}, and Soibelman~\cite{Soi16}. Thus the numerical condition used here should be viewed as the Hitchin-base counterpart of a standard condition in the existing literature.

 \item As noted in the introduction, the nilpotent case of the Deligne--Simpson problem was already solved by Kostov \cite{Kos03} and Crawley-Boevey~\cite{CB03}. Their results are in fact more comprehensive: provided $2r\le\sum_{i=1}^n\gamma_i^{1}$ and excluding several special types (occurring for $n=3$ or $n=4$; see \cite[Section~5]{CB03}), the nilpotent case admits irreducible solutions.
	
The approach presented here is more geometric and has the advantage of producing explicit solutions. It is natural to conjecture that even when the hypothesis of Lemma~\ref{integral lemma} is weakened, one can still find an element $\alpha\in\mathbf{H}_P$ whose associated spectral curve is integral.
 \end{itemize}
	
\end{rem}

The conditions in Lemma~\ref{integral lemma} can also be used to discuss properties of the quiver varieties.

\begin{cor}Let $Q$, $\chi$, and $\mathbf v$ be determined by the parabolic type \eqref{types associated to conjugacy classes}. If the conditions of Lemma~{\rm \ref{integral lemma}} are satisfied, then the quiver varieties $\mathfrak{M}_{\chi}(\mathbf{v})$ and $\mathfrak{M}_{0}(\mathbf{v})$ are non-empty and contain irreducible representations.
\end{cor}

\begin{proof} Consider the following composed map:
$\mathbf{Higgs}^{\circ}_P \xrightarrow{\Psi} \mathfrak{M}_{\chi}(\mathbf{v})\rightarrow \mathfrak{M}_{0}(\mathbf{v})$.
So it suffices to find a homologically trivial parabolic Higgs bundle whose image in $\mathfrak{M}_{0}(\mathbf{v})$ is irreducible. Since the conditions of Lemma~\ref{integral lemma} are satisfied, we have an integral spectral curve~$C_{\alpha}$ and a line bundle~$\mathcal{L}$ on \smash{$\tilde{C}_{\alpha}$} such that the pushforward gives a homologically trivial parabolic Higgs bundle $(E,\phi)$.

 We claim that the image of $(E,\phi)$ in $\mathfrak{M}_{0}(\mathbf{v})$ is an irreducible representation. The reason is the same as in the Lemma~\ref{inte}: if it admits a proper nonzero subrepresentation, then, by the description of $\Psi$, as in the last part of Section~\ref{section 3}, such a subrepresentation will give a~proper nonzero parabolic Higgs subbundle, then characteristic polynomial of $(E,\phi)$ would factor, contradicting the integrality of~$C_{\alpha}$.
\end{proof}

\begin{rem}In general, nonemptiness of stable representations in $\mathfrak M_\chi(\mathbf v)$ does not imply the existence of irreducible representations in the affine variety $\mathfrak M_0(\mathbf v)$. The point of the construction above is that the parabolic Higgs bundles obtained from line bundles on integral spectral curves are irreducible in the stronger sense: they have no proper nonzero Higgs subbundles. Therefore, their images give irreducible representations of $\mathfrak M_0(\mathbf v)$, not merely $\chi$-semistable representations of $\mathfrak M_\chi(\mathbf v)$.

So this also explains the phenomenon: even though we pick weights to construct the moduli space of parabolic Higgs bundles, our approach to solving the Deligne--Simpson problem has nothing to do with this choice.
\end{rem}

\begin{Example}\label{example of construction}
 In this example we consider the following nilpotent additive Deligne--Simpson problem in rank $3$. We want to construct four nilpotent matrices
$A_0, A_1, A_2, A_3\in \mathfrak{gl}_3$
with no common proper nonzero invariant subspace such that
$A_0+A_1+A_2+A_3=0$,
where the Jordan types are
$[2,1]$, $[3]$, $[3]$, $[3]$.
In this case, we compute
$2r=6< 1+2+2+2=\sum_{i=1}^4\gamma_i^{1}$.
So by Theorem~\ref{DSP}, we have irreducible solutions. In the following, we actually find a family of solutions.

Now we use $z$ to denote the standard coordinate on the affine chart $\mathbb{A}^1=\mathbb{P}^1\setminus \{\infty\}\subset \mathbb{P}^1$. We~take $I=\{0,1,2,3\}\subset \mathbb P^1$ and let $D=[0]+[1]+[2]+[3]$ be the corresponding effective divisor of~$\mathbb{P}^1$. So we consider a parabolic Hitchin system on $\mathbb{P}^1$ with the following parabolic type: at $z=0$, it is given by the flag type $\{2, 1\}$ and at $z=1,2,3$, we take the type of full flags.

So by Proposition~\ref{Hitchin base}, the Hitchin base is given by
\[
\mathbf{H}_P=\operatorname{H}^0\bigl(\mathbb{P}^1, \omega_{\mathbb{P}^1}\bigr)\oplus \operatorname{H}^0\bigl(\mathbb{P}^1, \omega_{\mathbb{P}^1}^{\otimes 2}\otimes \mathcal{O}_{\mathbb{P}^1}(D)\bigr)\oplus \operatorname{H}^0\bigl(\mathbb{P}^1, \omega_{\mathbb{P}^1}^{\otimes 3}\otimes \mathcal{O}_{\mathbb{P}^1}(2D-[0])\bigr).
\]
Thus $\dim \mathbf{H}_P=3$. We set $p(z)=z(z-1)(z-2)(z-3)$ and trivialize $\omega_{\mathbb{P}^1}(D)$ over $\mathbb{A}^1$ using $\frac{{\rm d}z}{p(z)}$. Hence over $\mathbb{A}^1$, the spectral curve associated to a point in $\mathbf{H}_P$ is defined by
\begin{equation}\label{defining equation}
 C=C_{A,q}\colon \ y^3+A p(z)y+z p(z)q(z)=0,
\end{equation}
where
$A\in \mathbb C$, $q(z)=bz+c$
is a polynomial of degree at most~$1$. We assume throughout that
\begin{equation}\label{conditions}
 A\neq 0,\qquad b\neq 0,\qquad q(0)q(1)q(2)q(3)\neq 0, \qquad 4A^3+27b^2\neq 0.
\end{equation}
If we denote $\eta=\frac{y}{z^2}$, then at $z=\infty\in \mathbb{P}^1$, the equation~\eqref{defining equation} becomes
\begin{equation}\label{equation at infinity}
 \eta^3 + A\eta + b=0.
\end{equation}
So the last condition in~\eqref{conditions} ensures that the fiber over $z=\infty$ consists of three distinct points.

We further assume that $A$, $b$, $c$ satisfy a general open condition such that the curve $C_{A,q}$ is smooth away from the points lying over $z=0$. This open condition is non-empty, as can be verified by the Jacobian criterion. We use \smash{$\widetilde{C}$} to denote the normalization of $C$, and let
\smash{$\widetilde\pi\colon \widetilde C\longrightarrow \mathbb P^1$}
be the natural map. By \cite[Corollary~1]{SWW22}, the genus of $\widetilde C$ is $3$. Now we need to find a line bundle $\widetilde L$ so that $\widetilde \pi_* \widetilde L\cong \mathcal{O}_{\mathbb{P}^1}^{\oplus 3}$. We still denote by~$z$ and $y$ their pullbacks to $\widetilde C$.

Let
$P_1, P_2, P_3\in \widetilde C$
be the points lying over $(z,y)=(1,0), (2,0), (3,0) \in C$, respectively. These are ramification points of index $3$. Let $\rho=-\frac{c}{b}$ be the zero of $q$. By the assumption $q(0)q(1)q(2)q(3)\neq0$, we have $\rho\notin \{0,1,2,3\}$. Let
\smash{$Q\in\widetilde C$} be the point lying over ${(z,y)=(\rho,0)\in C}$. We further denote
$R, U \in \widetilde C$
to be the points lying over $z=0\in \mathbb{P}^1$ so that~$R$ is the ramified branch with ramification index $2$ and $U$ is the unramified branch.

Now we consider the divisor
$D_q:=2P_3+P_1+R+Q$
and set
\smash{$\widetilde L=\mathcal{O}_{\widetilde C}(D_q)$}.
We define three meromorphic functions on~\smash{$\widetilde C$} by
\[
e_1=1,\qquad
e_2=\frac{z(z-2)}{y},\qquad
e_3=\frac{y}{z(z-3)},
\]
and we claim that those meromorphic functions form a basis of \smash{$\operatorname{H}^0\bigl(\widetilde C, \widetilde L\bigr)$} and moreover, the pushforward of $\widetilde L$ to $\mathbb{P}^1$ is homologically trivial.

Let $F_\infty=\widetilde\pi^{-1}(\infty)=\infty_1+\infty_2+\infty_3$ be the reduced divisor consisting of the three points lying over infinity. Then we have the divisor identities
\begin{gather*}
\operatorname{div}(z)=2R+U-F_\infty,
\\
\operatorname{div}(z-1)=3P_1-F_\infty,\qquad
\operatorname{div}(z-2)=3P_2-F_\infty,\qquad
\operatorname{div}(z-3)=3P_3-F_\infty,
\end{gather*}
and
\[
\operatorname{div}(y)=R+U+P_1+P_2+P_3+Q-2F_\infty.
\]
Then we can compute that
\[
\operatorname{Pole}(e_2)\le P_1+P_3+Q, \qquad \operatorname{Pole}(e_3)\le R+2P_3.
\]
Therefore,
\smash{$e_1,e_2,e_3\in \operatorname{H}^0\bigl(\widetilde C,\widetilde L\bigr)$}.
To show that these meromorphic functions are linearly independent, it suffices to show that they are independent at $F_\infty$. Let $\eta_1$, $\eta_2$ and $\eta_3$ be the $3$ distinct roots of the equation~\eqref{equation at infinity}, so they are all nonzero since $b\neq 0$. For each $j=1,2,3$, the values of the sections at $\infty_j\in F_\infty$ are
\[
e_1(\infty_j)=1,\qquad e_2(\infty_j)=\frac{1}{\eta_j},\qquad e_3(\infty_j)=\eta_j.
\]
Thus the evaluation matrix at $F_\infty$ is
\[
\begin{pmatrix}
1&1&1\\
1/\eta_1&1/\eta_2&1/\eta_3\\
\eta_1&\eta_2&\eta_3
\end{pmatrix},
\]
whose determinant is nonzero because the $\eta_j$'s are distinct and nonzero. Hence $e_1$, $e_2$, $e_3$ are linearly independent.
Hence using $e_1$, $e_2$, $e_3$, we have an injection
\smash{$\mathcal{O}_{\mathbb{P}^1}^{\oplus 3}\longrightarrow \widetilde \pi_* \widetilde L$},
which is an isomorphism since $\deg \widetilde \pi_* \widetilde L=0$.

Now we consider the action of $y$ on $e_1$, $e_2$ and $e_3$, we have
$y(e_1,e_2,e_3)=(e_1,e_2,e_3)B(z)$,
where
\[
B(z)=
\begin{pmatrix}
0&z(z-2)&-A(z-1)(z-2)\\
0&0&-q(z)(z-1)\\
z(z-3)&0&0
\end{pmatrix}.
\]
Thus we have
\begin{alignat*}{3}
 & B(0)=
\begin{pmatrix}
0&0&-2A\\
0&0&q(0)\\
0&0&0
\end{pmatrix}, \qquad && B(1)=
\begin{pmatrix}
0&-1&0\\
0&0&0\\
-2&0&0
\end{pmatrix},& \\
& B(2)=
\begin{pmatrix}
0&0&0\\
0&0&-q(2)\\
-2&0&0
\end{pmatrix},\qquad && B(3)=
\begin{pmatrix}
0&3&-2A\\
0&0&-2q(3)\\
0&0&0
\end{pmatrix}.&
\end{alignat*}
The residues of $\phi$ at $z=0,1,2,3$ are
\[
A_0=-\frac{1}{6}B(0),\qquad
A_1=\frac{1}{2}B(1),\qquad
A_2=-\frac{1}{2}B(2),\qquad
A_3=\frac{1}{6}B(3).
\]
These residues have the desired Jordan types. Finally, since $q(z)=bz+c$ is linear, the residue theorem gives
$A_0+A_1+A_2+A_3=0$.
Moreover, since $A_1$, $A_2$ and $A_3$ have Jordan type $[3]$, so each of them has exactly two proper nonzero invariant subspaces. None of them are identical. So $\{A_i\}_{i=0}^3$ is an irreducible solution.
\end{Example}

\subsection*{Acknowledgements} I am grateful to Jia Choon Lee for encouraging me to publish this work. I~would also like to thank Xiaotao Sun, Xiaoyu Su, Bin Wang and Bingyu Zhang for many helpful discussions and suggestions. I~am also grateful to the anonymous referees for their careful reading of the manuscript and for their valuable comments and suggestions.

X.~Wen is supported by the Chongqing Natural Science Foundation Innovation and Development Joint Fund (CSTB2023NSCQ-LZX0031 and CSTB2024NSCQ-LZX0092) and Chongqing University of Technology Research Start--up Funding Project 2023ZDZ013.

\pdfbookmark[1]{References}{ref}
\LastPageEnding

\end{document}